\documentclass[11pt]{article}

\usepackage{amsfonts,amsmath,amssymb}

\usepackage{ulem}

\usepackage[utf8]{inputenc}
\usepackage{newunicodechar}

\usepackage{hyperref}
\usepackage{url}
\usepackage[all,cmtip]{xy}

\usepackage{etoolbox}
\usepackage{amsthm}
\usepackage{tocloft}
\usepackage{semmac}

\usepackage{manfnt}

\newcommand{\conj}[1]{\iota(#1)}
\newcommand{\op}[1]{\operatorname{#1}}

\newcommand{\hyp}{Let $\ell$ be a prime as in \Cref{assump_degreePolarization}(1). }

\usepackage{cleveref}

\newtheorem{theorem}{Theorem}

\newtheorem{corollary}[theorem]{Corollary}

\newtheorem{lemma}[theorem]{Lemma}
\newtheorem{notation}[theorem]{Notation}

\newtheorem{proposition}[theorem]{Proposition}

\theoremstyle{definition}
\newtheorem{definition}[theorem]{Definition}
\newtheorem{remark}[theorem]{Remark}
\newtheorem{assumption}[theorem]{Assumption}

\numberwithin{theorem}{section}

\newcommand{\abGal}[1]{\operatorname{Gal}\left(\overline{#1}/#1\right)}
\newcommand{\Gq}{G^F}
\newcommand{\threefolds}

\usepackage{a4wide}

\date{}
\begin{document}

\title{Explicit open image theorems for abelian varieties with trivial endomorphism ring}

\author{Matthew Bisatt and Davide Lombardo}

\maketitle
\begin{abstract}
Let $K$ be a number field and $A/K$ be an abelian variety of dimension $g$. Assuming that the image $G_{\ell^\infty}$ of the natural Galois representation attached to the Tate module $T_\ell(A)$ is $\operatorname{GSp}_{2g}(\mathbb{Z}_\ell)$ for all sufficiently large primes $\ell$, we provide a semi-effective bound $\ell_0(A/K)$ such that $G_{\ell^\infty}=\operatorname{GSp}_{2g}(\mathbb{Z}_\ell)$ for all primes $\ell > \ell_0(A/K)$. The bound is given in terms of the Faltings height of $A$ and of the cardinality of the residue field at a suitably generic place of $K$. We also describe an algorithmic approach to obtain better bounds for abelian threefolds over $\mathbb{Q}$.
\end{abstract}

\setcounter{tocdepth}{1}
\tableofcontents

\newpage
\section{Introduction}
Let $K$ be a number field and $A$ be an abelian variety over $K$.
The purpose of the present work is to study the Galois representations attached to $A$ under the assumption that $\operatorname{End}_{\overline{K}}(A)$ is $\mathbb{Z}$. More precisely, we are interested in the family of representations
\[
\rho_{\ell^\infty} : \abGal{K} \to \operatorname{Aut} T_\ell A \cong \operatorname{GL}_{2g}(\mathbb{Z}_\ell)
\]
arising from the $\ell$-adic Tate modules of $A$. We will also consider the residual mod-$\ell$ representations
$
\rho_{\ell} : \abGal{K} \to \operatorname{Aut}(A[\ell]) \cong \operatorname{GL}_{2g}(\mathbb{F}_\ell),
$
and write $G_{\ell^{\infty}}$ (resp.~$G_\ell$) for the image of $\rho_{\ell^\infty}$ (resp.~of $\rho_\ell$). By work of many authors, in particular Serre \cite{Serre_resum8586} and Pink \cite{Pink}, much is known about these representations: if -- for some prime $\ell$ -- the Zariski closure of $G_{\ell^\infty}$ has rank $g+1$, then the same is true for all primes, and in fact the equality $G_{\ell^\infty}=\operatorname{GSp}_{2g}(\mathbb{Z}_\ell)$ holds for all $\ell$ large enough (depending on $A/K$). If we assume the Mumford--Tate conjecture, this condition is equivalent to the Mumford--Tate group of $A$ being $\operatorname{GSp}_{2g,\mathbb{Q}}$. Unconditionally, if $g=\dim A$ lies outside a certain very thin set the condition $\operatorname{End}_{\overline{K}}(A)=\mathbb{Z}$ implies $\operatorname{MT}(A)=\operatorname{GSp}_{2g,\mathbb{Q}}$ and the Mumford--Tate conjecture is true for $A$ (cf.~\Cref{thm_Pink} below).

Our objective is to make these results explicit, by finding a bound $\ell_0(A/K)$ (which, as the notation suggests, will depend on $A$ and $K$) such that, for all primes $\ell>\ell_0(A/K)$, the representation $\rho_{\ell^\infty}$ surjects onto $\operatorname{GSp}_{2g}(\mathbb{Z}_\ell)$. 
To state our results more compactly we introduce the following functions:

\begin{definition}\label{def_bFunction}
Let $A/K$ be a $g$-dimensional abelian variety over a number field $K$. Let $h(A)$ be its stable Faltings height, with the original normalisation of Faltings \cite{MR740897} (see also \cite[§2.3]{MR3225452}). We define
\[
b(A/K) = b([K:\mathbb{Q}],g,h(A))=\left( (7g)^{8g^2} [K:\mathbb{Q}] \max\left(h(A), \log [K:\mathbb{Q}],1 \right) \right)^{2g^2}
\]
and $b(A/K; d) = b(d[K:\mathbb{Q}],g,h(A))$. 

\end{definition}

Our main result is the following explicit bound for the quantity $\ell_0$ introduced above:

\begin{theorem}[=\Cref{thm_MainProof}]\label{thm_Main}
Let $A/K$ be an abelian variety of dimension $g \geq 2$ and $G_{\ell^\infty}$ be the image of the natural representation $\rho_{\ell^\infty} : \abGal{K} \to \operatorname{Aut} T_\ell A$. Suppose that:
\begin{enumerate}
\item $\operatorname{End}_{\overline{K}}(A)=\mathbb{Z}$;
\item there exists a place $v$ of $K$, of good reduction for $A$ and with residue field of order $q_v$, such that the characteristic polynomial of the Frobenius at $v$ (cf.~\Cref{sect_FrobElem}) has Galois group $\left(\mathbb{Z}/2\mathbb{Z}\right) \wr \op{Sym}_g$.
\end{enumerate}
Then the equality $G_{\ell^\infty}=\operatorname{GSp}_{2g}(\mathbb{Z}_\ell)$ holds for every prime $\ell$ unramified in $K$ and such that
\[
\ell > \max\left\{ \left(2q_v^{4(\log_2(2g)\sqrt[3]{2g} + 1)}\right)^{2^g\cdot g!}, \; b(A/K; 2g), \right\}.
\]
\end{theorem}

\begin{remark}\label{rmk: genus 2}
The case $g=1$ is treated in detail in \cite{MR1209248} and \cite{AdelicEC}, while the case $g=2$ is considered in \cite{Surfaces} and in the Appendix of \cite{MR3981312}.
\cite{Surfaces} uses a semistability above $\ell$ assumption but has no dependence on $q_v$; on the other hand the appendix of \cite{MR3981312} uses a bound of $(2q_v)^8$. This is better than the bound we give above but a case-by-case analysis of our results recovers this exponent when $g=2$. 
\end{remark}

Let us make a few comments on condition (2) in \Cref{thm_Main}. The hypothesis concerning the Galois group of $f_v(x)$ might seem somewhat unnatural, but it is a simple way to encode the fact that the roots of $f_v(x)$ are ``maximally generic'': this condition implies that the subgroup of $\overline{\mathbb{Q}}^\times$ they generate is free of rank $g+1$ \cite[Theorem 1.7 and Remark 1.3]{MR4735944}. 
We also remark that if at least one suitable place $v$ exists, then the set of such places has density one; furthermore, the existence of such a place is \textit{equivalent} to the fact that the equality $G_{\ell^\infty}=\operatorname{GSp}_{2g}(\mathbb{Z}_\ell)$ holds for some $\ell$. Thus, if the equality $G_{\ell^\infty}=\operatorname{GSp}_{2g}(\mathbb{Z}_\ell)$ holds for all sufficiently large primes $\ell$ (in particular, if the conclusion of the theorem is true for \textit{some} bound $\ell_0(A/K)$), it is very easy in practice to find a place $v$ that satisfies condition (2). 
From a more theoretical perspective, a theorem of Pink, combined with earlier work of Serre, implies that such places $v$ always exist if the dimension of $A$ lies outside a certain (very thin) set of ``exceptional'' dimensions:
\begin{theorem}[{Pink \cite{Pink}, Serre \cite{Serre_resum8586,LettreVigneras}}]\label{thm_Pink}
Let $A$ be an abelian variety of dimension $g$ defined over the number field $K$ and let
\[
S=\left\{ \frac{1}{2}(2n)^k \bigm\vert n>0, k\geq 3 \text{ odd} \right\} \cup \left\{\frac{1}{2} \binom{2n}{n} \bigm\vert n>1 \text{ odd}\right\}.
\]
If $\operatorname{End}_{\overline{K}}(A)=\mathbb{Z}$ and $\dim A \not \in S$, then:
\begin{itemize}
\item the Mumford--Tate conjecture is true for $A$;
\item the Mumford--Tate group of $A$ is isomorphic to $\operatorname{GSp}_{2g,\mathbb{Q}}$;
\item the equality $G_{\ell^\infty}=\operatorname{GSp}_{2g}(\mathbb{Z}_\ell)$ holds for every sufficiently large prime $\ell$.
\end{itemize}
\end{theorem}

\begin{remark}
The set $S$ is closely related to the set of dimensions $d$ for which there exists a symplectic minuscule representation of dimension $d$, cf.~\Cref{sect_LieGroups}.

\end{remark}

To have a completely effective result, we would also need to show that the number $q_v$ above can be effectively bounded \textit{a priori} in terms of simple arithmetical invariants of $A/K$. Under GRH, this has been achieved to a large extent in \cite{ZywinaEffectiveOpenImage} (which also gives an unconditional result similar in spirit to \Cref{thm_Main}). However, our techinques are different, and offer two main advantages: on the one hand, we get a completely explicit result, whereas \cite[Theorem 1.2]{ZywinaEffectiveOpenImage} involves some unspecified (although in principle effective) constants. On the other, as we describe in \S \ref{sect: CompBounds}, the method of proof also leads to a general algorithm to compute a good upper bound on the largest non-surjective prime of an abelian threefold $A/\mathbb{Q}$ with trivial geometric endomorphism ring, with large parts of the approach applying to abelian varieties of arbitrary dimension over arbitrary number fields. This generalises both \cite{MR1969642, MR4732686} (which consider abelian surfaces over $\mathbb{Q}$) and \cite{ZywinaExample} (which deals with semistable abelian threefolds over $\mathbb{Q}$ having one prime where the reduction has toric rank $1$).

To conclude this introduction we now describe the organisation of this paper. After two sections of preliminaries (§ \ref{sect_Preliminaries} and \ref{sect_GroupTheory}) we study the various classes of maximal proper subgroups $G$ of $\operatorname{GSp}_{2g}(\mathbb{F}_\ell)$, showing that -- at least for $\ell$ large enough -- $G_\ell$ cannot be contained in any such $G$. This occupies \S \ref{sect_EasyCases}-\ref{sect_LieGroups}, each of which deals with a different kind of maximal subgroup, with \Cref{thm_Main} finally proved in §\ref{sect_MainProof}.

Since the bound given by \Cref{thm_Main} can be prohibitively large, we take a more practical (albeit less uniform) approach in \S \ref{sect: CompBounds}; this is illustrated by an example in \S \ref{sect: Example} which could not be analysed by previous methods. 

We say a few more words on the techniques used in Sections \ref{sect_EasyCases} through \ref{sect_LieGroups}.
Three classes of maximal subgroups (traditionally dubbed ``imprimitive", ``reducible", and ``field extension" cases) are dealt with in \S\ref{sect_EasyCases} as an almost immediate consequence of the isogeny theorem of Masser and W\"ustholz \cite{MR1217345,MR1207211} (the completely explicit version we employ is due to Gaudron and Rémond \cite{PolarisationsEtIsogenies,MR4592752}). 

In §\ref{sect_TensorProductsI} we consider the case of $G_\ell$ being contained in a ``tensor product" subgroup, and we show how, given a place $v$ as in hypothesis (2) of \Cref{thm_Main}, one can produce a finite set of integers whose divisors include all the primes for which $G_\ell$ is of tensor product type. This is inspired by an argument of Serre \cite{LettreVigneras}, but his use of the characteristic polynomial of $\operatorname{Fr}_v$ is almost completely replaced by a direct study of the multiplicative relations satisfied by its roots. 

In \S\ref{sect_ConstantGroups}, we eliminate the possibility of $G_\ell$ being a small ``exceptional" (or ``constant") group by obtaining a lower bound on $\#\mathbb{P}G_\ell$ via a result of Collins. Finally, in  \S\ref{sect_LieGroups} we use tools from representation theory (of both algebraic and finite groups) to treat the last remaining class of maximal subgroups of $\operatorname{GSp}_{2g}(\mathbb{F}_\ell)$, namely the so-called exceptional subgroups of class $\mathcal{S}$, once again forcing relations $\bmod{\, \ell}$ between eigenvalues of Frobenius elements as in \S\ref{sect_TensorProductsI}.

\smallskip

\noindent\textbf{Acknowledgments.} The second author thanks N. Ratazzi for suggesting the problem and J.-P. Serre for inspiring discussions. The authors are supported by MUR grant PRIN-2022HPSNCR (funded by the European Union project Next Generation EU).

\section{Preliminaries}\label{sect_Preliminaries}

\subsection{The isogeny theorem and the Weil pairing}\label{sect_IsogenyTheorem}
Many of our estimates rely on the following `isogeny theorem', originally due to Masser and W\"ustholz \cite{MR1207211,MR1217345} and then improved and made explicit by Gaudron and Rémond. The theorem applies to arbitrary abelian varieties, but we restrict to those with trivial endomorphisms which has an improved bound and is the only situation we need.

\begin{theorem}[Isogeny Theorem, {\cite[Theorem 1.9(1)]{MR4592752}}]\label{thm_Isogeny}
Let $A/K$ be an abelian variety such that $\operatorname{End}_K(A)=\mathbb{Z}$. For every abelian variety $A^*$ defined over $K$ that is $K$-isogenous to $A$, there exists a $K$-isogeny $A^* \to A$ whose degree is bounded by $b(A/K)$ (cf.~\Cref{def_bFunction}). 
\end{theorem}

\vspace{5pt}

An immediate application of the isogeny bound is to polarising $A$. We fix now (once and for all) a polarisation $\varphi$ of $A$ of minimal degree; this satisfies $\deg \varphi \leq b(A/K)$. With this fixed, we are able to construct the corresponding Weil pairing on the $\ell$-adic Tate module of $A$. Moreover, if $\ell$ is coprime to $\deg \varphi$, then the Weil pairing is nondegenerate so $G_{\ell^{\infty}} \subseteq \op{GSp}_{2g}(\mathbb{Z}_\ell)$ and $G_{\ell}$ is a subgroup of $\operatorname{GSp}_{2g}(\mathbb{F}_\ell)$. Combining this with \cite[Theorem 4.1]{MR2081941} (cf.~also \cite[Lemma 1 on p.~52]{LettreVigneras}) we have:

\begin{lemma}\label{cor_Conclusion}

Suppose that $A$ admits a polarisation of degree prime to $\ell$: then $G_\ell \subseteq \operatorname{GSp}_{2g}(\mathbb{F}_\ell)$; this is in particular true for all $\ell>b(A/K)$. If $\dim A=1$ suppose $\ell \geq 5$, and if $\dim A =2$ suppose $\ell \geq 3$. If $K$ is linearly disjoint from $\mathbb{Q}(\zeta_{\ell^\infty})$ (for example, if $\ell$ is unramified in $K$), the inclusion $\operatorname{Sp}_{2g}(\mathbb{F}_\ell) \subseteq G_\ell$ implies $G_{\ell^\infty}=\operatorname{GSp}_{2g}(\mathbb{Z}_\ell)$.
\end{lemma}

\subsection{Frobenius elements and their eigenvalues}\label{sect_FrobElem}

Let $\Omega_K$ denote the set of finite places of $K$. For each $v \in \Omega_K$, we write $q_v$ for the cardinality of the residue field at $v$. We also write $\operatorname{Fr}_v \in \abGal{K}$ for a Frobenius element at $v$.
If $v$ is a place of $K$ of good reduction for $A$, the characteristic polynomial of $\rho_{\ell^\infty}(\operatorname{Fr}_v)$ does not depend on $\ell$ (as long as $v \nmid \ell$), and will be denoted by $f_v(x) \in \mathbb{Z}[x]$. We write $\mu_1,\ldots,\mu_{2g}$ for the roots of $f_v(x)$ in $\overline{\mathbb{Q}}$ and call these algebraic integers the \textbf{eigenvalues of $\operatorname{Fr}_v$}. Their sum is called the \textbf{trace of $\operatorname{Fr}_v$}, denoted by $\operatorname{tr} \operatorname{Fr}_v$.

The splitting field of $f_v(x)$ is a Galois extension of $\mathbb{Q}$ which we call $F(v)$. By the Weil conjectures (proven in general by Deligne \cite{MR0340258}), the absolute value of every $\mu_i$ under any embedding of $F(v)$ in $\mathbb{C}$ is equal to $q_v^{1/2}$. The restriction of complex conjugation to $F(v)$ gives a central element $\iota \in \operatorname{Gal}(F(v)/\mathbb{Q})$ of order at most $2$ (this element acts as $\mu \mapsto q_v / \mu$ on the roots of $f_v(x)$).
If $\ell$ is a rational prime not lying below $v$, choose a place $\mathfrak{l} \mid \ell$ of $F(v)$. We denote the reductions $\bmod{\mathfrak{l}}$ of the algebraic integers $\mu_i$ by $\overline{\mu}_i$ and shall implicitly identify these with the roots in $\overline{\mathbb{F}}_\ell$ of the characteristic polynomial of $\rho_\ell(\op{Fr}_v)$.

\begin{notation}
Let $v \in \Omega_K$ be a place of good reduction for $A$. We denote by $\Phi_v = \{\mu_1, \ldots, \mu_{2g} \}$ the set of roots of $f_v(x)$ and by $\Gamma_v$ the subgroup of $\overline{\mathbb{Q}}^\times$ generated by $\Phi_v$. We order the $\mu_i$ so that $\mu_{2g+1-i} = \iota(\mu_i)$ for all $1 \leq i \leq2g$, where $\iota$ is as above; this implies $\mu_i\mu_{2g+1-i}=q_v$ for all $1 \leq i \leq 2g$. We denote by $\Gamma_v^{\iota=\operatorname{id}}$ the subgroup of $\Gamma_v$ fixed by $\iota$ and by $\overline{\mu}_i$ the reduction of $\mu_i$ in $\overline{\mathbb{F}}_\ell$.
\end{notation}

\begin{remark}
Note that since $\mu_i\mu_{2g+1-i}=q_v \in \mathbb{Q}$ for all $i$, the Galois action on $\Phi_v$ preserves the partition given by the pairs $\{\mu_i, \mu_{2g+1-i}\}$ for $1 \leq i \leq g$. Consequently, the Galois group of $f_v(x)$ can be identified with a subgroup of the wreath product $\mathbb{Z}/2\mathbb{Z} \, \wr \, \op{Sym}_g$, where $\op{Sym}_g$ acts by permuting the pairs $\{\mu_i, \mu_{2g+1-i}\}$ and the $i$-th copy of $\mathbb{Z}/2\mathbb{Z}$ acts as a transposition on $\{\mu_i, \mu_{2g+1-i}\}$. In particular, $[F(v):\mathbb{Q}]$ divides $2^g g!$.
\end{remark}

\begin{remark}\label{rem:gen_sets}
The relations $\mu_i\mu_{2g+1-i}=q_v$ impose that the subgroup $\Gamma_v= \langle \mu_1, \ldots, \mu_{2g} \rangle$ of $\overline{\mathbb{Q}}^\times$ has rank at most $g+1$, with generating sets $\{ \mu_1, \ldots, \mu_g, \mu_{g+1} \}$ and $\{ \mu_1, \ldots, \mu_g, q_v \}$. Moreover, if $\Gamma_v$ has rank $g+1$, it is free.
\end{remark}

From now on, we'll mostly work with places $v$ such that the subgroup $\Gamma_v$ generated by Frobenius eigenvalues has maximal rank which is a very mild constraint (cf.~the discussion before \Cref{thm_Pink}).

\begin{lemma}\label{lemma_EigenvaluesGenerateFreeGroup}
Suppose that $g \geq 2$ and $\Gamma_v$ has rank $g+1$. Let $(n_1,\ldots,n_{2g}) \in \mathbb{Z}^{2g}$ be such that
$\prod_{i=1}^{2g} \mu_i^{n_i} \in \Gamma_v^{\iota = \operatorname{id}}.$  Then $n_i=n_{2g+1-i}$ for all $1 \leq i \leq 2g$. In particular, $\Gamma_v^{\iota = \operatorname{id}} = q_v^\mathbb{Z}$.
\end{lemma}

\begin{proof}
Using $\mu_i\iota(\mu_i)=q_v$, we rewrite $\alpha:=\prod_{i=1}^{2g} \mu_i^{n_i}$ as $q_v^m\prod_{i=1}^g \mu_i^{n_i-n_{2g+1-i}},$ where $m=\sum_{i=g+1}^{2g} n_i.$ On the other hand, $\iota(\alpha)=q_v^{m'} \prod_{i=1}^g \mu_i^{n_{2g+1-i}-n_i}$ where $m'=\sum_{i=1}^g n_i$.
We rearrange the equality $\alpha=\iota(\alpha)$ to get $\prod_{i=1}^g \mu_i^{2(n_i-n_{2g+1-i})} = q_v^{m'-m}$; since all these terms are multiplicatively independent (cf. \Cref{rem:gen_sets}), we necessarily have $n_i-n_{2g+1-i}=0$ for all $1 \leq i \leq 2g$.
\end{proof}

\begin{lemma}\label{lemma_Squares}
Suppose that $g \geq 2$ and $\Gamma_v$ has rank $g+1$. Let $\lambda_1, \lambda_2, \lambda_3, \lambda_4$ be four distinct eigenvalues of $\operatorname{Fr}_v$.

\begin{enumerate}
\item We have $\lambda_1^2 \neq \lambda_2\lambda_3$.
\item If $\lambda_1\lambda_2=\lambda_3\lambda_4$, then $\lambda_2=\conj{\lambda_1}$ and $\lambda_4=\conj{\lambda_3}$.
\end{enumerate}
Let $m, N$ be positive integers with $m$ odd and $\lambda_1, \ldots, \lambda_{2m}$ be distinct eigenvalues of $\operatorname{Fr}_v$.
\begin{enumerate}
\setcounter{enumi}{2}
\item We have $\prod_{i=1}^m \lambda_i^N \neq \prod_{i=m+1}^{2m} \lambda_i^N$.
\end{enumerate}

\end{lemma}

\begin{proof}
(1): Note that any three eigenvalues are multiplicatively independent; this follows immediately from \Cref{rem:gen_sets}.

(2): Observe first that $\{\lambda_1,\lambda_2,\lambda_3,\lambda_4\}=\{\alpha,\beta,\iota(\alpha),\iota(\beta)\}$ for some eigenvalues $\alpha,\beta$ (otherwise the set is multiplicatively independent by \Cref{rem:gen_sets}).
Suppose for contradiction that $\iota(\lambda_1) \neq \lambda_2$. The equation becomes $\lambda_1\lambda_2=\iota(\lambda_1\lambda_2)$ and hence $\lambda_1\lambda_2 \in \Gamma_v^{\iota=\operatorname{id}}$. However by \Cref{lemma_EigenvaluesGenerateFreeGroup} we have $\Gamma_v^{\iota=\operatorname{id}}=q_v^{\mathbb{Z}}$ and hence $\lambda_1,\lambda_2,q_v$ are not multiplicatively independent which is a contradiction.

(3): Note that the equality $\prod_{i=1}^m \lambda_i^N = \prod_{i=m+1}^{2m} \lambda_i^N$ implies that $\prod_{i=1}^m \lambda_i^N\prod_{i=m+1}^{2m} \lambda_i^{-N} = 1$ belongs to $\Gamma_v^{\iota = \operatorname{id}}$. By \Cref{lemma_EigenvaluesGenerateFreeGroup}, the number of exponents equal to $N$ in the product $\prod_{i=1}^m \lambda_i^N\prod_{i=m+1}^{2m} \lambda_i^{-N}$ is even, which contradicts the assumption that $m$ is odd and $N>0$.
\end{proof}

The main application of this section will be through the following proposition which will enable us to bound $\ell$ whenever we are able to force relations between the eigenvalues of $\rho_\ell(\operatorname{Fr}_v)$.

\begin{proposition}\label{prop_MultRelBoundsEll}
Suppose that $g \geq 2$ and $\Gamma_v$ has rank $g+1$. Let $\lambda_1,\ldots,\lambda_n$ be distinct eigenvalues of $\operatorname{Fr}_v$ and let $\ell$ be a rational prime such that $v \nmid \ell$. Let $N$ be a positive integer.

\begin{enumerate}
\item If there exists a place $\mathfrak{l} \mid \ell$ of $F(v)$ such that $\lambda_2^{2N} \equiv \lambda_1^N\lambda_3^N \bmod{\mathfrak{l}}$, then $\ell \leq (2q_v^N)^{[F(v):\mathbb{Q}]}$.
\item If there exists a place $\mathfrak{l} \mid \ell$ of $F(v)$ such that $(\lambda_1\lambda_3)^N \equiv (\lambda_2\lambda_4)^N \bmod{\mathfrak{l}}$ and $(\lambda_1\lambda_6)^N \equiv (\lambda_2\lambda_5)^N \bmod{\mathfrak{l}}$, then $\ell \leq (2q_v^N)^{[F(v):\mathbb{Q}]}.$ 
\end{enumerate}
Let $m$ be a positive odd integer with $2m \leq n$. 
\begin{enumerate}
\setcounter{enumi}{2}
\item If there exists a place $\mathfrak{l} \mid \ell$ of $F(v)$ such that 
$(\lambda_1 \ldots \lambda_m)^N \equiv (\lambda_{m+1} \ldots \lambda_{2m})^N \bmod{\mathfrak{l}}$, then $\ell \leq (2q_v^{mN/2})^{[F(v) : \mathbb{Q}]}$.
\end{enumerate}
\end{proposition}

\begin{proof}
First note that by the Weil conjectures $|\sigma(\lambda_i^N\lambda_j^N)|=q^N_v$ for every pair of eigenvalues, Galois automorphism $\sigma$ and positive integer $N$. Hence, by the triangle inequality, the norm $|N_{F(v)/\mathbb{Q}}(\lambda_i^N\lambda_j^N-\lambda_k^N\lambda_l^N)|$ is a positive rational integer bounded by $(2q_v^N)^{[F(v)/\mathbb{Q}]}$. Each congruence condition implies that these norms are divisible by $\ell$ so our bounds (1) and (2) will then follow so long as the relevant norms are nonzero.

For (1), it follows immediately from \Cref{lemma_Squares}(1) that $|N_{F(v)/\mathbb{Q}}(\lambda_2^{2N}-\lambda_1^N\lambda_3^N)| \neq 0$ since $\Gamma_v$ is free.
For (2), we consider instead $|N_{F(v)/\mathbb{Q}}((\lambda_1\lambda_3)^N-(\lambda_2\lambda_4)^N)|$ and $|N_{F(v)/\mathbb{Q}}((\lambda_1\lambda_6)^N-(\lambda_2\lambda_5)^N)|$; they cannot both be zero, since otherwise $\lambda_3=\lambda_6=\iota(\lambda_1)$ by \Cref{lemma_Squares}(2) and the fact that $\Gamma_v$ is free, contradicting distinctness of the eigenvalues.
For (3) the argument is very similar, using $|\sigma(\lambda_1 \cdots \lambda_m)^N| \leq q_v^{mN/2}$ for each Galois automorphism $\sigma$ and \Cref{lemma_Squares}(3).
\end{proof}

\begin{remark}
    Note that the proof of \Cref{prop_MultRelBoundsEll} produces a nonzero integer divisible by $\ell$, so we don't just obtain an upper bound on $\ell$, but also a divisibility condition. We will exploit this is \Cref{sect: CompBounds}.
\end{remark}

\begin{lemma}\label{lemma: semisimple Frob}
Let $v$ be a place of $K$ of good reduction for $A$ such that $f_v(x)$ is squarefree. Let $\ell > (2\sqrt{q_v})^{[F(v):\mathbb{Q}]}$ be a prime such that $v \nmid \ell$. Then $\rho_\ell(\operatorname{Fr}_v)$ is semisimple with distinct eigenvalues.
\end{lemma}

\begin{proof}
It suffices to prove that $\rho_\ell(\operatorname{Fr}_v)$ has distinct eigenvalues. If not, then similarly to the proof of \Cref{prop_MultRelBoundsEll}, $\ell$ divides the nonzero norm $|N_{F(v)/\mathbb{Q}}(\lambda_1-\lambda_2)| \leq (2\sqrt{q_v})^{[F(v):\mathbb{Q}]}$ for some pair of distinct roots $\lambda_1, \lambda_2$ of $f_v$.
\end{proof}

We conclude this section by recording our working assumption that $\ell$ does not divide the degree of a minimal polarisation. This is a minor technical point, but is nonetheless necessary if we want our statements (which often involve the group $\operatorname{GSp}_{2g}(\mathbb{F}_\ell)$) to be meaningful. Since we will frequently use relations between eigenvalues to eliminate some cases, we will also fix a place $v$ that gives rise to a suitable Frobenius element.

\begin{assumption}\label{assump_degreePolarization}
The two assumptions we shall frequently impose are as follows:
\begin{enumerate}
\item The rational prime $\ell$ does not divide the degree of a minimal polarisation of $A$, hence we may identify $G_\ell$ (resp.~$G_{\ell^\infty}$) with a subgroup of $\operatorname{GSp}_{2g}(\mathbb{F}_\ell)$ (resp.~$\operatorname{GSp}_{2g}(\mathbb{Z}_\ell)$).
\item We fix a place $v$ of $K$ of good reduction for $A$ such that $\Gamma_v$ is free of rank $g+1$.

\end{enumerate}
\end{assumption}

\begin{remark}
Our primary reason to invoke the choice of $v$ is to rule out a particular class of subgroups via \Cref{prop_MultRelBoundsEll} for sufficiently large $\ell$; the lower bound placed on $\ell$ then implies that $v$ and $\ell$ are coprime. We shall therefore also suppose that $v \nmid \ell$ in our proofs without further reference.
\end{remark}

\section{Maximal subgroups of \texorpdfstring{$\operatorname{GSp}_{2n}(\mathbb{F}_\ell)$}{finite symplectic groups}}\label{sect_GroupTheory}

\newcommand{\GroupElement}{h}

To prove the main result that $G_{\ell^{\infty}}=\operatorname{GSp}_{2g}(\mathbb{Z}_\ell)$ for $\ell$ greater than our explicit bound, it suffices to prove the equality $G_\ell=\operatorname{GSp}_{2g}(\mathbb{F}_\ell)$  by \Cref{cor_Conclusion}. To do this, we will use a description of the maximal proper subgroups of $\operatorname{GSp}_{2g}(\mathbb{F}_\ell)$: the core of our argument will be to show that -- for $\ell$ large enough -- $G_\ell$ cannot be contained in any proper subgroup of $\operatorname{GSp}_{2g}(\mathbb{F}_\ell)$, and so it must coincide with $\operatorname{GSp}_{2g}(\mathbb{F}_\ell)$. The purpose of this section is to introduce some notation and provide the classification of maximal subgroups in \Cref{thm_Aschbacher}. Our main references for this section are \cite{MR3098485} and \cite{MR1057341}.

\subsection{Group theoretical preliminaries}

We begin by recalling some basic facts from finite group theory.

\begin{definition}
Let $G$ be a finite group. The \textbf{socle} of $G$, denoted $\operatorname{soc}(G)$, is the subgroup of $G$ generated by the nontrivial minimal normal subgroups of $G$. We call $G$  \textbf{almost simple} if $S=\operatorname{soc}(G)$ is nonabelian simple; in this case $S \leq G \leq \operatorname{Aut}(S)$ and $S$ is a normal subgroup of $G$.
\end{definition}

\begin{lemma}\label{lemma_NoNormalAbelianSubgroups}
An almost simple group $G$ does not possess nontrivial normal solvable subgroups.
\end{lemma}
\begin{proof}
Suppose a nontrivial normal solvable subgroup exists. Then the collection of such subgroups is nonempty, and there is a minimal nontrivial normal subgroup $N_0$ of $G$ that is solvable (a subgroup of a solvable group is itself solvable). The definition of $\operatorname{soc}(G)$ implies $N_0 \subset \operatorname{soc}(G)$, and moreover $N_0$ is normal in $\operatorname{soc}(G)$ since it is normal in $G$. By simplicity of $\operatorname{soc}(G)$ this forces $N_0=\operatorname{soc}(G)$; however, the latter is simple nonabelian, hence in particular not solvable, contradiction.
\end{proof}

\begin{definition}
Let $G$ be a finite group. The group $\operatorname{Inn}(G)$ of inner automorphisms of $G$ is the normal subgroup of the automorphism group generated by the conjugation maps. The quotient is the \textbf{outer automorphism group} $\operatorname{Out}(G).$
\end{definition}

\begin{definition}
A group is \textbf{perfect} if it equals its commutator subgroup. If $H$ is a finite group we denote by $H^{\infty}$ the first perfect group contained in the derived series of $H$; equivalently,
\[
\displaystyle H^\infty = \bigcap_{i\geq 0} H^{(i)},
\]
where $H^{(0)}=H$ and $H^{(i+1)}=[H^{(i)},H^{(i)}]$ for every $i \geq 0$.
\end{definition}

\begin{lemma}\label{lemma_socGIsPerfect}
Let $G$ be a finite almost simple group and let $N$ be the exponent of $\operatorname{Out}(\operatorname{soc}(G))$. Then:
\begin{itemize}
\item $\operatorname{soc}(G)=G^{\infty}$, in particular $\operatorname{soc}(G)$ is perfect;
\item $\GroupElement^N \in \operatorname{soc}(G)$ for all $\GroupElement \in G$.
\end{itemize}

\end{lemma}

\begin{proof}
Let $S=\operatorname{soc}(G)$.
By definition we have $S \leq G \leq \operatorname{Aut}(S)$, and therefore $S^\infty \subseteq G^\infty \subseteq \operatorname{Aut}(S)^\infty$. Note $S^\infty=S \cong \operatorname{Inn}(S)$ (as $S$ is simple and nonabelian), and $\operatorname{Aut}(S)^\infty / \operatorname{Inn}(S) \subseteq \left( \operatorname{Aut}(S)/\operatorname{Inn}(S) \right)^\infty = \operatorname{Out}(S)^\infty$. Thus, it suffices to prove that $\operatorname{Out}(S)^\infty$ is trivial.
This follows immediately from the fact that the outer automorphism group of a simple group is solvable \cite[Theorem 1.3.2]{MR3098485}.
For the last statement, by definition of exponent, for every $\varphi \in \operatorname{Aut}(S)$ we have $\varphi^N \in \operatorname{Inn} (S) \cong S$.
From $G \subseteq \operatorname{Aut}(S)$ we then obtain $\GroupElement^N \in \{\varphi^N : \varphi \in \operatorname{Aut(S)}\} \subseteq S$.
\end{proof}

\subsection{Classical definitions}
We now recall some standard definitions in the theory of finite matrix groups, in particular the symplectic groups, Kronecker product, and semilinear groups. Throughout this section, we let $F$ be a finite field of characteristic $\ell \geq 5$. We denote by $\operatorname{GL}_n(F)$ the group of $n \times n$ invertible matrices with coefficients in $F$. For a matrix $x \in \operatorname{GL}_n(F)$, we denote by ${}^t x$ the transpose.
For every subgroup $G$ of $\operatorname{GL}_n(F)$, we denote by $\mathbb{P}G$ the image of $G$ in the quotient $\displaystyle \operatorname{\mathbb{P}GL}_n(F) := \frac{\operatorname{GL}_n(F)}{F^\times \cdot \operatorname{Id}}$. 

\begin{definition}
Let $J$ denote the antidiagonal matrix $\operatorname{antidiag}(\underbrace{1,\ldots,1}_n,\underbrace{-1,\ldots,-1}_n)$.
The \textbf{symplectic group} and the \textbf{group of symplectic similarities} are:
\begin{itemize}
\item $\operatorname{Sp}_{2n}(F)=\left\{ x \in \operatorname{GL}_{2n}(F) \bigm\vert {}^tx J x=J \right\};$
\item $\operatorname{GSp}_{2n}(F)=\left\{ x \in \operatorname{GL}_{2n}(F) \bigm\vert \exists \lambda \in F^\times \mbox{ such that } {}^tx J x=\lambda J \right\}.$
\end{itemize}
\end{definition}

We also recall the Kronecker product. 

\begin{definition}
Let $V_1, V_2$ be $F$-vector spaces. The Kronecker product of two endomorphisms $g_1 \in \operatorname{GL}(V_1)$ and $g_2 \in \operatorname{GL}(V_2)$ is the endomorphism $g_1 \otimes g_2$ of $V_1 \otimes_F V_2$ which acts as \[(g_1 \otimes g_2)(v_1 \otimes v_2) = (g_1 v_1) \otimes (g_2 v_2)\] on decomposable tensors, for all $v_1 \in V_1$ and $v_2 \in V_2$.
For integers $m,n \geq 1$ and for subgroups $G_1$ and $G_2$ of $\operatorname{GL}_m(F)$, $\operatorname{GL}_n(F)$ respectively, we write $G_1 \otimes G_2$ for the quotient of $G_1 \times G_2$ by the equivalence relation 
\[
(a,b) \sim (c, d)\text{ if and only if there exists }\lambda \in F^\times\text{ such that }c=\lambda a, \; d=\lambda^{-1} b.
\]
The group $G_1 \otimes G_2$ is in a natural way a subgroup of $\operatorname{GL}_{mn}(F)$, the inclusion being given by identifying $(g_1, g_2) \in G_1 \times G_2 / \sim$ with $g_1 \otimes g_2 \in \operatorname{GL}_{mn}(F)$: the equivalence relation $\sim$ ensures that this identification is well defined \cite[Proposition 1.9.8]{MR3098485}.
\end{definition}

Finally, we review the notion of semilinear map:
\begin{definition}\label{def: semilinear}
Let $V$ be an $F$-vector space. A map $f: V \rightarrow V$ is \textit{semilinear} if there exists a field automorphism $\theta \in \op{Aut}(F)$ such that
\[
f(\lambda_1v_1+\lambda_2v_2)=\theta(\lambda_1)f(v_1)+\theta(\lambda_2)f(v_2) \qquad \text{ for all } v_1,v_2 \in V, \,\, \lambda_1,\lambda_2 \in F.
\]
The general semilinear group of $V$ is the group of all invertible semilinear maps of $V$; it is isomorphic to $\op{GL(V)} \rtimes \op{Aut}(F).$ If $F=\mathbb{F}_{\ell^s}$, $s \geq 1$, then there is an embedding of the semilinear group into $\op{GL}_{s\dim V}(\mathbb{F}_\ell)$.
\end{definition}

\subsection{Maximal subgroups of $\operatorname{GSp}_{2n}(\mathbb{F}_\ell)$}
We now recall the classification of the maximal subgroups of $\operatorname{GSp}_{2n}(\mathbb{F}_\ell)$ for $\ell \geq 3$. 
We start with the notion of $m$-decomposition:
\begin{definition}[{cf.~\cite[§2.2.2]{MR3098485} and \cite[§4.2]{MR1057341}}]
Let $\ell$ be an odd prime and $m \geq 2$ be an integer. An $m$-decomposition of $\mathbb{F}_\ell^{2n}$ is the data of $m$ subspaces $V_1, \ldots, V_m$ of $\mathbb{F}_\ell^{2n}$, each of dimension $\frac{2n}{m}$, such that $\mathbb{F}_\ell^{2n} \cong \bigoplus_{i=1}^m V_i$.
\end{definition}

Before stating the classification theorem we need to define some of the \textbf{Aschbacher classes}. Class $\mathcal{C}_3$ concerns whether the group is semilinear over some field extension; the full power of this will not be needed for the proof of \Cref{thm_Main} so we only give the relevant property below, but this description will come in useful in \S \ref{sect: CompBounds} when we take an algorithmic approach.

\begin{definition}
\label{def_excepClasses}
A maximal subgroup $G$ of $\operatorname{GSp}_{2n}(\mathbb{F}_\ell)$ is said to be:
\begin{enumerate}
\item \textbf{reducible}, or of \textbf{class $\mathcal{C}_1$}, if it stabilises a linear subspace of $\mathbb{F}_\ell^{2n}$.
\item \textbf{imprimitive}, or of \textbf{class $\mathcal{C}_2$}, if there exists an $m$-decomposition $V_1,\ldots,V_m$ which is stable under the action of $G$ (i.e., for all $\GroupElement \in G$ and for all $1 \leq i \leq m$ there exists a $1 \leq j \leq m$ such that $\GroupElement V_i \subseteq V_j$) and the restriction of the standard symplectic form of $\mathbb{F}_\ell^{2n}$ to $V_i$ is either nondegenerate for every $1 \leq i \leq m$, or trivial for every $1 \leq i \leq m$.
\item \textbf{a field extension subgroup}, or of \textbf{class $\mathcal{C}_3$}, if it satisfies \cite[Definition 2.2.5]{MR3098485}. 
In this case, there exists a prime $s \mid 2n$ and a structure of $\mathbb{F}_{\ell^s}$-vector space on $\mathbb{F}_\ell^{2n}$ such that $G$ acts $\mathbb{F}_{\ell^s}$-semilinearly on $\mathbb{F}_{\ell}^{2n}$. Moreover, $G$ admits an index-$s$ subgroup $H$ that acts $\mathbb{F}_{\ell^s}$-linearly.

\item a \textbf{tensor product subgroup}, or of \textbf{class $\mathcal{C}_4$}, if there is a decomposition $\mathbb{F}_\ell^{2n} \cong V_1 \otimes V_2$ (where $V_1, V_2$ are $\mathbb{F}_\ell$-vector spaces of respective dimensions $n_1$ and $n_2 \geq 3$) such that for each $\GroupElement \in G$ there exist $\GroupElement_1 \in \operatorname{GL}(V_1)$ and $\GroupElement_2 \in \operatorname{GL}(V_2)$ that satisfy $\GroupElement=\GroupElement_1 \otimes \GroupElement_2$. Moreover, the action of $G$ on $V_1$, respectively $V_2$, preserves a symplectic, respectively symmetric, bilinear form up to scalars.
\item a \textbf{tensor induced subgroup}, or of \textbf{class $\mathcal{C}_7$}, if there exist positive integers $m, t$ with $2n=(2m)^t$ and a decomposition $\mathbb{F}_\ell^{2n} \cong V_1 \otimes V_2 \otimes \cdots \otimes V_t$ (where $V_1, \ldots, V_t$ are $\mathbb{F}_\ell$-vector spaces all of the same dimension $2m$) with the following property: for all $\GroupElement \in G$, there exists a permutation $\sigma \in \op{Sym}_t$ and operators $\GroupElement_1,\ldots,\GroupElement_t \in \operatorname{GSp}_{2m}(\mathbb{F}_\ell)$ such that \[\GroupElement(v_1 \otimes \cdots \otimes v_t)=(\GroupElement_{\sigma(1)}v_{\sigma(1)}) \otimes \cdots \otimes (\GroupElement_{\sigma(t)}v_{\sigma(t)}) \quad  \text{ for all } v_1 \in V_1, \ldots, v_t \in V_t. \] Such a $G$ is isomorphic to $\operatorname{GSp}_{2m}(\mathbb{F}_\ell)^{\otimes t} \rtimes \op{Sym}_t$.
\end{enumerate}
\end{definition}

We shall also have to deal with the exceptional class $\mathcal{S}$:
\begin{definition}[{cf.~\cite[Definition 2.1.3]{MR3098485}}]\label{def_ClassS}
A subgroup $G$ of $\operatorname{GSp}_{2n}(\mathbb{F}_\ell)$ is said to be of class $\mathcal{S}$ if and only if all of the following hold:
\begin{enumerate}
\item $\mathbb{P}G$ is almost simple;
\item $G$ does not contain $\operatorname{Sp}_{2n}(\mathbb{F}_\ell)$;
\item $G^{\infty}$ acts absolutely irreducibly on $\mathbb{F}_\ell^{2n}$;
\item $G^\infty$ preserves a nonzero symplectic form, but no nonzero unitary or quadratic forms.
\end{enumerate}
\end{definition}

\begin{remark}\label{rmk: constant groups vs Lie groups}
It is a general philosophy (cf.~for example \cite{Serre_resum8586}, especially §3, or \cite[Remark 2.1]{MR1969642}) that groups in class $\mathcal{S}$ should come in two different flavours. On one hand, there should exist finitely many groups that embed as maximal subgroups in $\mathbb{P}\operatorname{GSp}_{2g}(\mathbb{F}_\ell)$ for infinite families of primes $\ell$; we shall call them the \textbf{constant} (or \textbf{exceptional}) groups. On the other hand, there also exist maximal subgroups in class $\mathcal{S}$ \textbf{``of Lie type''}, obtained as follows. Given an algebraic group $\mathcal{G}$ over $\mathbb{Z}$ admitting an absolutely irreducible, symplectic representation of dimension $2n$, one can consider the corresponding map $\varphi : \mathcal{G} \xrightarrow{} \operatorname{GSp}_{2n,\mathbb{Z}}$ and the subgroup $\varphi\left(\mathcal{G}(\mathbb{F}_\ell)\right)$ of $\operatorname{GSp}_{2g}(\mathbb{F}_\ell)$: for sufficiently large $\ell$, all the maximal subgroups in class $\mathcal{S}$ that are not ``constant'' should be normalisers of groups of the form $\varphi\left(\mathcal{G}(\mathbb{F}_\ell)\right)$ for suitable $\mathcal{G}$ and $\varphi$ (and the groups $\mathcal{G}$ involved should be independent of $\ell$). 
We do not turn these notions into precise definitions, but it will be clear from \Cref{sect_ConstantGroups,sect_LieGroups} that there are indeed two different kinds of class $\mathcal{S}$ subgroups, and that we analyse them in different ways.
\end{remark}

\smallskip

We are now ready to state the following classification theorem, essentially due to Aschbacher (but see also \cite[Main Theorem and Table 3.5.C]{MR1057341} and \cite[§3]{MR779398}):

\begin{theorem}[{Aschbacher \cite{MR746539}}]\label{thm_Aschbacher}
Let $n$ be a positive integer, $\ell \geq 3$ be a prime, and $G$ be a maximal proper subgroup of $\operatorname{GSp}_{2n}(\mathbb{F}_\ell)$ not containing $\operatorname{Sp}_{2n}(\mathbb{F}_\ell)$. One of the following holds:
\begin{enumerate}
\item $G$ is of class $\mathcal{C}_1$;
\item $G$ is of class $\mathcal{C}_2$, stabilising an $m$-decomposition for some $m \geq 2$ dividing $2n$;
\item $G$ is of class $\mathcal{C}_3$ for some prime $s$ dividing $2n$;
\item $G$ is of class $\mathcal{C}_4$, and more precisely $G$ is contained in $\operatorname{GSp}_{2m}\left(\mathbb{F}_\ell \right) \otimes B$ where $B \subseteq \operatorname{GL}_t(\mathbb{F}_\ell)$ preserves a nondegenerate symmetric bilinear form up to scalars (so $B$ is bilinear: see \Cref{def: bilinear}), where $m$ and $t \geq 3$ are integers such that $2mt=2n$; 
\item $G$ is of class $\mathcal{C}_6$ with $2n=2^m$, and  $[\mathbb{P}G:\Omega_{2m}^{-}(2)] \in \{2^{2m+1}, 2^{2m+2},2^{2m+3}\}$. 
\item $G$ is of class $\mathcal{C}_7$ for some pair $(m,t)$ such that $(2m)^t=2n$ and $t \geq 3$ is odd;
\item $G$ is of class $\mathcal{S}$.
\end{enumerate}
\end{theorem}

\begin{remark}
Aschbacher's theorem is actually a much more general statement, giving a classification of the maximal subgroups of all finite classical groups, but in the present work we only need the case of $\operatorname{GSp}_{2n}(\mathbb{F}_\ell)$.
\end{remark}

\begin{proof}
We use Aschbacher's theorem \cite[Theorem 2.2.19]{MR3098485} with the descriptions of the classes and notation contained in \textit{loc.~cit.}~(whilst their focus is on $n\leq 6$, one can check that the classification holds for all $n$ \cite{MR1057341}, and \cite[§3]{MR779398}). In our case we are considering case $\mathbf{S}$ with $\Omega = \operatorname{Sp}_{2n}(\mathbb{F}_\ell)$ and $A$ contains the conformal group $\operatorname{GSp}_{2n}(\mathbb{F}_\ell)$. The theorem then states that any maximal subgroup of $\operatorname{GSp}_{2n}(\mathbb{F}_\ell)$ is contained in one of the geometric classes $\mathcal{C}_1, \ldots, \mathcal{C}_8$ or the exceptional class $\mathcal{S}$ given above.

Note the classes $\mathcal{C}_1, \mathcal{C}_2, \mathcal{C}_3$ and $\mathcal{C}_7$ are precisely as described in \Cref{def_excepClasses} (for class $\mathcal{C}_7$, see \cite[Table 2.10]{MR3098485} for the condition $t \equiv 1 \pmod 2$). Classes $\mathcal{C}_5$ and $\mathcal{C}_8$ cannot occur since $\ell$ is prime and odd respectively so it remains to justify the descriptions for classes $\mathcal{C}_4$ and $\mathcal{C}_6$.

For $\mathcal{C}_4$, one immediately finds from \cite[Table 2.7]{MR3098485} that $G \cap \operatorname{Sp}_{2n}(\mathbb{F}_\ell) \cong \operatorname{Sp}_{2m}\left(\mathbb{F}_\ell \right) \otimes \operatorname{GO}^{\varepsilon}_t\left( \mathbb{F}_\ell \right)$ with $2mt=2n$, $t \geq 3$, where $\operatorname{GO}^{\varepsilon}_t\left( \mathbb{F}_\ell \right)$ are various orthogonal subgroups of $\operatorname{GL}_t\left( \mathbb{F}_\ell \right)$ with $\varepsilon \in \{+,-,\circ \}.$ The corresponding maximal subgroups of $\operatorname{GSp}_{2n}( \mathbb{F}_\ell)$ are precisely the groups $\operatorname{GSp}_{2m}\left(\mathbb{F}_\ell \right) \otimes \operatorname{CGO}^{\varepsilon}_t\left( \mathbb{F}_\ell \right)$.

Finally we treat class $\mathcal{C}_6$ which necessitates that $2n$ is a prime power $2^m$. By the classification, the group $\mathbb{P}G$ is an extension of degree at most $2$ of a maximal subgroup $H < \mathbb{P}\operatorname{Sp}_{2g}(\mathbb{F}_\ell)$ of class $\mathcal{C}_6$. By \cite[Proposition 4.6.9]{MR1057341} or \cite[Table 2.9]{MR3098485}, the group $H$ is an extension of degree at most $2$ of $2^{2m+1}.\Omega^-_{2m}(2)$.

\end{proof}

\begin{remark}\label{rmk_OrdBizarreGroup}
We will not give a precise definition of the orthogonal group $\Omega_{2m}^{-}(2)$; all we need to know to prove \Cref{thm_Main} is that its order is given by (\cite[Theorem 1.6.22 and Table 1.3]{MR3098485})
\[
\#\Omega_{2m}^{-}(2) = 2^{m^2-m}(2^m+1)\prod_{i=1}^{m-1}(2^{2i}-1).
\]
In particular, in class $\mathcal{C}_6$ above we have
\[
\#\mathbb{P}G \leq 2^{2m+3} \cdot 2^{m^2-m} (2^m+1) \cdot \prod_{i=1}^{m-1}(2^{2i}-1) \leq (2^m+5)! =(2n+5)!.
\]

\end{remark}

The proof of \Cref{thm_Main} essentially consists in going through the list provided by \Cref{thm_Aschbacher} to show that, for $\ell$ large enough, $G_\ell$ cannot be contained in any proper maximal subgroup of $\operatorname{GSp}_{2g}(\mathbb{F}_\ell)$, and therefore the equality $G_\ell=\operatorname{GSp}_{2g}(\mathbb{F}_\ell)$ must hold.

\section{Reducible, imprimitive and field extension cases}\label{sect_EasyCases}
We now begin our elimination of the possible maximal subgroups. Recall from the introduction that we denote by $A/K$ an abelian variety of dimension $g$ with $\operatorname{End}_{\overline{K}}(A)=\mathbb{Z}$, and by $G_\ell$ the image of the representation \[\rho_\ell : \abGal{K} \to \operatorname{Aut} A[\ell] \cong \operatorname{GL}_{2g}(\mathbb{F}_\ell).\]
For $\ell > b(A/K)$, we know from \Cref{cor_Conclusion} that (up to a choice of basis) we have $G_\ell \subseteq \operatorname{GSp}_{2g}(\mathbb{F}_\ell)$. Suppose that $G_\ell$ does not contain $\operatorname{Sp}_{2g}(\mathbb{F}_\ell)$: then $G_\ell$ is contained in one of the maximal subgroups listed in \Cref{thm_Aschbacher}. The following proposition shows that $G$ cannot be of class $\mathcal{C}_1$, $\mathcal{C}_2$ or $\mathcal{C}_3$ for $\ell$ large enough:

\begin{proposition}\label{prop_EasyCases} Let $K$ be a number field and $A/K$ be a $g$-dimensional abelian variety such that $\operatorname{End}_{\overline{K}}(A)=\mathbb{Z}$. \hyp Let $G$ be a maximal proper subgroup of $\operatorname{GSp}_{2g}(\mathbb{F}_\ell)$ that contains $G_\ell$. Suppose $G$ is of class $\mathcal{C}_1, \mathcal{C}_2$ or $\mathcal{C}_3$: then $\ell \leq b(A/K; 2g)$, where $b$ is the function of \Cref{def_bFunction}.
\end{proposition}

\begin{proof}

Replacing $K$ with an extension $L$ of degree at most $2g$ we can assume that the Galois action stabilises a nontrivial subspace of $A[\ell]$ (classes $\mathcal{C}_1$ and $\mathcal{C}_2$), or that it acts on $A[\ell]$ by preserving an $\mathbb{F}_{\ell^s}$-linear structure for some $s > 1$ (class $\mathcal{C}_3$).
If the Galois action over $L$ preserves a nontrivial subspace, we obtain $\ell \leq b(A/L)$ by \cite[Lemma 3.17]{Surfaces} and \Cref{thm_Isogeny}.
If the image of Galois preserves an $\mathbb{F}_{\ell^s}$-structure, its centraliser in $\operatorname{End}_{\mathbb{F}_\ell}(A[\ell])$ strictly contains $\mathbb{F}_\ell$. In particular, the cokernel of the natural map $\mathbb{Z} = \operatorname{End}_L(A) \to \operatorname{End}_{\operatorname{Gal}(\overline{L}/L)}(A[\ell])$ has order divisible by $\ell$. By \cite[Theorem 1.9(4)]{MR4592752} we obtain $\ell \leq b(A/L) \leq b(A/K; 2g)$. 
\end{proof}

\section{The tensor product case}\label{sect_TensorProductsI}
We now consider the problem of showing that, for sufficiently large $\ell$, the group $G_\ell$ cannot be contained in a tensor product or tensor induced subgroup. 
Let us briefly explain the key idea behind the proof, which goes back to Serre (cf.~\cite{LettreVigneras}), using the notation introduced in §\ref{sect_FrobElem}.
If $G_\ell$ is contained in a tensor product subgroup, this forces the eigenvalues of every $\GroupElement \in G_\ell$ to satisfy a number of additional multiplicative relations besides the ``obvious'' ones arising from being an element of $\operatorname{GSp}_{2g}(\mathbb{F}_\ell)$, which are not generally present. Applying this to the image of a suitable Frobenius element $\GroupElement=\rho_\ell(\operatorname{Fr}_v)$, we can use these extra relations to bound $\ell$ via \Cref{prop_MultRelBoundsEll}.

We first make a definition that does not seem to appear in the literature but will suit our purposes. In particular, this definition will encompass the general symplectic and orthogonal groups that we use. With this in hand, we then prove a simple lemma that encapsulates the precise property we want.

\begin{definition}\label{def: bilinear}
We call a subgroup $B \subseteq \op{GL}_n(\mathbb{F}_\ell)$ \textit{bilinear} if it preserves a nondegenerate bilinear form up to scalars.
\end{definition}

\begin{lemma}\label{lem_involution}
Let $n\geq 1$ and let $B \subseteq \op{GL}_n(\mathbb{F}_\ell)$ be a bilinear subgroup. Then for all elements $b \in B$, there exists $\kappa \in \mathbb{F}_\ell^\times$ such that the multiset of eigenvalues of $b$ is invariant under the involution $\nu \mapsto \kappa/\nu$.
\end{lemma}

\begin{proof} Let $M$ be the matrix representing the bilinear form.
By assumption we have ${}^tb M b = \kappa \cdot M$ for some scalar $\kappa \in \mathbb{F}_\ell^\times$, hence $\kappa b^{-1} = M^{-1} \cdot {}^t b \cdot M$. Since the characteristic polynomial of a matrix is invariant under conjugation and transposition, we find that $\kappa b^{-1}$ and $b$ have the same characteristic polynomial. The claim follows.
\end{proof}

\begin{proposition}\label{prop:tensor products}
Let $\ell$ and $v$ be as in \Cref{assump_degreePolarization}. Let $n\geq 3$ and suppose that $B \subseteq \op{GL}_n(\mathbb{F}_\ell)$ is a bilinear subgroup. If $G_\ell$ is a contained in a tensor product group $\op{GSp}_{2m}(\mathbb{F}_\ell) \otimes B$, then $\ell \leq (2q_v)^{[F(v):\mathbb{Q}]}$.
\end{proposition}

\begin{proof}
Suppose for contradiction that $\ell > (2q_v)^{[F(v):\mathbb{Q}]}$. Up to conjugacy, for every element $\GroupElement$ of $G_\ell$ there exist operators $a \in \operatorname{GSp}_{2m}(\mathbb{F}_\ell)$ and $b \in B$ such that $\GroupElement=a \otimes b$. Let $\nu_1, \nu_2$ be eigenvalues of $a$ and let $\phi$ be the involution (given by \Cref{lem_involution}) on the the multiset $\Lambda$ of eigenvalues of $b$.

If $\Lambda$ contains a fixed point $\lambda$ of $\phi$, then for all $\phi$-stable triples $(\lambda, \lambda_1, \lambda_2)$ of eigenvalues of $b$, there exists an indexing such that $\{x_1,x_2,x_3 \} = \{ \nu_1\lambda, \nu_1\lambda_1, \nu_1\lambda_2 \}$ are eigenvalues of $\GroupElement$ satisfying $x_2^2=x_1x_3$.

If $\Lambda$ contains no fixed points, then $n \geq 4$ and let $\lambda_1, \lambda_2$ be eigenvalues of $b$ such that $\lambda_2 \neq \phi(\lambda_1)$. Set $x_1=\nu_1\lambda_1, x_2=\nu_1\lambda_2, x_3=\nu_2\phi(\lambda_1), x_4=\nu_2\phi(\lambda_2), x_5=\nu_2\lambda_1, x_6=\nu_2\lambda_2$ and observe that these are eigenvalues of $\GroupElement$ satisfying the relations
\[
x_1x_{3}=x_2x_4 \qquad \text{and} \qquad x_1x_6=x_2x_5.
\]
Now let $\GroupElement=\rho_\ell(\operatorname{Fr}_v)$ and in both cases lift the eigenvalues $x_i$ to distinct eigenvalues of $\op{Fr}_v$ by \Cref{lemma: semisimple Frob}. The contradiction then follows from \Cref{prop_MultRelBoundsEll}(1) and (2) respectively.
\end{proof}

\subsection{Tensor induced groups}

In the tensor product situation, we were able to force relations directly between the eigenvalues for arbitrary group elements. We now slightly generalise this notion to deal with the case when we don't have relations in the entire group but only in a subgroup of a certain shape. The subgroups we will choose have the property that a bounded power (depending only on $g$) of every group element lies in such a subgroup.

We will apply this statement in two ways: for semidirect products (for the tensor induced case); and via the outer automorphism group (in the exceptional class setting).

\begin{proposition}\label{prop_NoELargerThan1}
Let $\ell$ and $v$ be as in \Cref{assump_degreePolarization}. Let $H$ be a subgroup of $\operatorname{GSp}_{2g}(\mathbb{F}_\ell)$ having a subgroup $J \triangleleft H$ contained in $\operatorname{GSp}_{2m}(\mathbb{F}_\ell) ^{\otimes t}$ where $2g=(2m)^t$, $t\geq 3$. Let $N \geq 1$ be such that, for every $h \in H$, there exists $n$ with $1 \leq n \leq N$ such that $h^n \in J$. If $G_\ell$ is contained in $H$, then $\ell \leq (2q_v^N)^{[F(v) : \mathbb{Q}]}$.

\end{proposition}

\begin{proof}
Fix $h \in H$ and let $n$ be as in the statement. We have $y := h^n \in J \subseteq \operatorname{GSp}_{2m}(\mathbb{F}_\ell) ^{\otimes t}$. Note that {$\operatorname{GSp}_{2m}(\mathbb{F}_\ell) ^{\otimes t} = \operatorname{GSp}_{2m}(\mathbb{F}_\ell) \otimes \operatorname{GSp}_{2m}(\mathbb{F}_\ell)^{\otimes (t-1)}$ where the second factor is symplectic (resp.~orthogonal) if $t$ is even (resp.~odd).} 
As in the proof of \Cref{prop:tensor products} with $B=\operatorname{GSp}_{2m}(\mathbb{F}_\ell)^{\otimes (t-1)} \subset \op{GL}_{g/m}(\mathbb{F}_\ell)$, we then see that $y$ has either eigenvalues $x_1, x_2, x_3$ such that $x_2^2=x_1x_3$, or eigenvalues $x_1, \ldots, x_6$ such that $x_1x_3=x_2x_4$ and $x_1x_6=x_2x_5$. We apply this to $h = \rho_\ell(\operatorname{Fr}_v)$: the $x_i$ are $n$-th powers of eigenvalues $\overline{\mu}_i$ of $\rho_\ell(\operatorname{Fr}_v)$, and we obtain
\[
\overline{\mu}^{2n} = \overline{\mu}_1^n\overline{\mu}_3^n \quad\text{ or }\quad \begin{cases}
    \overline{\mu}_1^n\overline{\mu}_3^n = \overline{\mu}_2^n\overline{\mu}_4^n \\
    \overline{\mu}_1^n\overline{\mu}_6^n = \overline{\mu}_2^n\overline{\mu}_5^n.
\end{cases}
\]
Since $n \leq N$, lifting the $\overline{\mu}_i$ to distinct roots of $f_v(x)$ and applying \Cref{prop_MultRelBoundsEll}(1) and (2) concludes the proof.
\end{proof}

We apply \Cref{prop_NoELargerThan1} immediately to provide a bound for $G_\ell$ to be contained in a tensor induced subgroup using the function 
\begin{equation}\label{eq: definition of the function xi}
    \xi(t):=\max_{\sigma \in \op{Sym}_t} \operatorname{order} (\sigma),
\end{equation}
which will be bounded later.

\begin{proposition}\label{prop: NoInducedTensors rewrite}
Let $\ell$ and $v$ be as in \Cref{assump_degreePolarization}. Suppose $\ell > (2q_v^{\xi(\lfloor \log_2(2g)\rfloor)})^{[F(v):\mathbb{Q}]}$.
Then the group $G_\ell$ is not contained in a tensor induced subgroup $\operatorname{GSp}_{2m}(\mathbb{F}_\ell)^{\otimes t} \rtimes \op{Sym}_t$, where $m \geq 1, t \geq 3$ are integers such that $(2m)^t=2g$.
\end{proposition}

\begin{proof}
Apply \Cref{prop_NoELargerThan1} with $H=\operatorname{GSp}_{2m}(\mathbb{F}_\ell)^{\otimes t} \rtimes \op{Sym}_t$ and $J=\operatorname{GSp}_{2m}(\mathbb{F}_\ell) ^{\otimes t}$. By definition, we can take $N=\xi(t)$. Since $2g=(2m)^t$, we have $t \leq \lfloor \log_2(2g) \rfloor$, and we are done. 
\end{proof}

\section{Constant groups in class \texorpdfstring{$\mathcal{S}$}{S}} \label{sect_ConstantGroups}

In this section, we consider maximal subgroups of class $\mathcal{S}$ whose projective image has order bounded independently of $\ell$ (we call these the \textit{constant} groups in class $\mathcal{S}$, see \Cref{rmk: constant groups vs Lie groups}). The analysis of the constant subgroups of $\operatorname{GSp}_{2g}(\mathbb{F}_\ell)$ is greatly simplified by a theorem of Collins, but prior to stating it we first define the subgroup $E_\ell(G)$.

\begin{definition}
Let $G$ be a group and $\ell$ a prime. We define two characteristic subgroups of $G$ as follows:
\begin{itemize}
   \item The Bender subgroup $E(G)$ which is generated by subnormal subgroups of $G$ that are quasisimple; it is the central product of such subgroups.
   \item The subgroup $E_\ell(G)$  generated by normal perfect subgroups $E$ of $E(G)$ such that $E/Z(E)$ is a simple group of Lie type in characteristic $\ell$.
\end{itemize}
\end{definition}

\begin{theorem}[{\cite[Theorem B]{MR2456628}}]\label{thm: Collins}
Let $G \subseteq \operatorname{GL}_{n}(\mathbb{F}_\ell)$ be a primitive subgroup. There exists an explicit constant $J(n)$ such that
$[G:Z(G)\cdot E_{\ell}(G)] \leq J(n)$. Moreover, one can take $J(n)=(n+5)!$.
\end{theorem}

\begin{remark}
Collins' result is actually sharper in general: for instance, if $n \geq 13$ and $\ell \nmid (n+2)$, then one can take $J(n)=(n+1)!$.
\end{remark}

Since we have already dealt with imprimitive subgroups in \S\ref{sect_EasyCases}, we may restrict to the primitive setting now in order to apply the above theorem.

\begin{proposition}\label{prop_ConstantOrLie}
Suppose that $G \subseteq \operatorname{GSp}_{2g}(\mathbb{F}_\ell)$ is a maximal primitive subgroup of class $\mathcal{S}$ that satisfies $\#\mathbb{P}G > J(2g)$. The socle of $\mathbb{P}G$ is a simple group of Lie type in characteristic $\ell$.
\end{proposition}

\begin{proof}
Define the subgroups $\Gamma_1=Z(G) \cdot E_\ell(G)$ and $\Gamma_2=Z(G) \cdot Z(E_\ell(G))$. Observe that since $E_\ell(G)$ is a characteristic subgroup of $G$, both $\Gamma_1$ and $\Gamma_2$ are also characteristic and hence normal in $G$. Moreover $\Gamma_1/\Gamma_2 \cong E_\ell(G)/Z(E_\ell(G))$; since the Bender subgroup is a central product of quasisimple groups, $\Gamma_1/\Gamma_2$ is a direct product of simple groups of Lie type in characteristic $\ell$ from definitions. Note that $\mathbb{P}\Gamma_2$ is trivial (it is a normal abelian subgroup of $\mathbb{P}G$; use \Cref{lemma_NoNormalAbelianSubgroups} and the fact that $\mathbb{P}G$ is almost simple by definition of class $\mathcal{S}$) hence $\mathbb{P}\Gamma_1$ is a normal subgroup of the almost simple group $\mathbb{P}G$. Moreover our hypotheses imply that $\mathbb{P}\Gamma_1$ is nontrivial since  $\#\mathbb{P}G \leq J(2g) \cdot \#\mathbb{P}\Gamma_1$ by \Cref{thm: Collins}.

Since $\operatorname{soc} (\mathbb{P}G)$ is the unique minimal normal subgroup of $\mathbb{P}G$, we have $\operatorname{soc} (\mathbb{P}G) \subseteq \mathbb{P}\Gamma_1 \subseteq \mathbb{P}G$. This implies $\operatorname{soc} (\mathbb{P}G) = \left(\operatorname{soc} (\mathbb{P}G) \right)^\infty \subseteq \left(\mathbb{P}\Gamma_1\right)^\infty \subseteq \left(\mathbb{P}G\right)^\infty = \operatorname{soc} (\mathbb{P}G)$ by \Cref{lemma_socGIsPerfect}, so in particular, $\operatorname{soc} (\mathbb{P}G) = \left(\mathbb{P}\Gamma_1\right)^\infty$. On the other hand, $\mathbb{P}\Gamma_1$ is a product of simple nonabelian groups, hence is perfect, so $\left(\mathbb{P}\Gamma_1\right)^\infty = \mathbb{P}\Gamma_1 = \operatorname{soc} (\mathbb{P}G)$ is simple by definition. Thus, $\mathbb{P}\Gamma_1$ is a simple group that is a direct product of finite simple groups of Lie type in characteristic $\ell$: it follows immediately that $\mathbb{P}\Gamma_1$ is itself a finite simple group of Lie type in characteristic $\ell$.
\end{proof}

As a consequence, we can easily show that $G_\ell$ is not contained in a constant group of class $\mathcal{S}$ for $\ell$ larger than an explicit bound. We formulate this by saying that, for $\ell$ large enough, if $G_\ell$ is contained in a maximal subgroup $G$ of class $\mathcal{S}$, then $\operatorname{soc} (\mathbb{P}G)$ is a simple Lie group in characteristic $\ell$ (and therefore not a constant group).

\begin{proposition}\label{prop_UnconditionalConstantGroups}
\hyp For every positive integer $N$ we have $\#\mathbb{P}G_\ell > N$ so long as $\ell$ is strictly larger than $b(A/K;N)^{1/(2g-1)}$. If $\ell$ is larger than $b(A/K;J(2g))^{1/(2g-1)}$, and if $G$ is a maximal subgroup of $\operatorname{GSp}_{2g}(\mathbb{F}_\ell)$ of class $\mathcal{S}$ such that $G_\ell \subseteq G$ and $G$ is primitive, then $\mathbb{P}G$ is an almost simple group with socle of Lie type in characteristic $\ell$.
\end{proposition}

\begin{proof}
Let $K'$ be the extension of $K$ which is the fixed field of $\ker\left(\abGal{K} \to \mathbb{P}G_\ell \right)$. By construction, the image of $\abGal{K'}$ in $\operatorname{Aut} A[\ell]$ consists entirely of scalars, so over $K'$ the representation $A[\ell]$ admits an invariant subspace of dimension 1 (in fact, any subspace will do). Repeating the argument of \cite[Lemma 3.17]{Surfaces} we find the desired result (in the notation of that proof we have $\#H=\ell$).
The second part of the statement then follows from \Cref{prop_ConstantOrLie}, since under these hypotheses we have $\#\mathbb{P}G \geq \#\mathbb{P}G_\ell > J(2g)$.
\end{proof}

\newcommand{\MaxOrderOut}{{4(r+1)}}
\newcommand{\TwiceMaxOrderOut}{{8(r+1)}}
\newcommand{\ThriceMaxOrderOut}{{12(r+1)}}
\newcommand{\ThreeHalvesMaxOrderOut}{{6(r+1)}}
\newcommand{\HalfMaxOrderOut}{{2(r+1)}}

\section{Class \texorpdfstring{$\mathcal{S}$}{S} subgroups of Lie type}\label{sect_LieGroups}
In this section we study the maximal subgroups of $\operatorname{GSp}_{2g}(\mathbb{F}_\ell)$ of class $\mathcal{S}$ ``of Lie type''. More precisely, in view of the result of \Cref{prop_ConstantOrLie}, we are interested in the maximal class $\mathcal{S}$ subgroups $H$ of $\operatorname{GSp}_{2g}(\mathbb{F}_\ell)$ such that $\operatorname{soc}\left(\mathbb{P}H\right)$ is a simple group of Lie type in characteristic $\ell$.
We will study these groups by comparing their representation theory with that of certain algebraic groups, so we begin by reviewing some generalities about the latter.
From now on, to avoid the pathologies associated with finite Suzuki and Ree groups, we assume $\ell \neq 2,3$.

Throughout this section we let $G/\overline{\mathbb{F}}_\ell$ be a simple, simply connected algebraic group of rank $r$ in characteristic $\ell \geq 5$. We fix a maximal torus $T$ of $G$ and write $\Lambda \cong \mathbb{Z}^r$ for its character group and $\left\{\alpha_1,\ldots,\alpha_r \right\} \subset \Lambda$ for its simple roots. There is an inner product $(\cdot,\cdot)$ on the vector space $\Lambda \otimes \mathbb{R}$, well-defined up to rescaling, that makes it into a Euclidean space.

\subsection{Preliminaries on algebraic groups and root systems}

\begin{definition}\label{def_weights}
For $G$ as above we use the following definitions:
\begin{itemize}
    \item An element $\lambda \in \Lambda$ is called a \textbf{weight}.
    \item The \textbf{fundamental weights} $\omega_1,\cdots,\omega_r$ are the dual basis of $\{\alpha_i^\vee\}_{1 \leq i \leq r}$, where  $\alpha^\vee:= \frac{2\alpha}{(\alpha,\alpha)}$; these form an integral basis of $\Lambda$ since $G$ is simply connected.

\item We call a weight $\lambda=\sum_{i=1}^r a_i\omega_i$ \textbf{dominant} if $a_i\geq 0$ for all $1 \leq i \leq r$.
    \item Given a positive integer $m$, a weight $\lambda=a_1\omega_1+ \cdots+a_r\omega_r$ is \textbf{$m$-restricted} if $0 \leq a_i \leq m-1$ for all $1\leq i \leq r$.

\item There is a partial ordering on $\Lambda$: $\lambda \succeq \mu$ if $\lambda - \mu$ is a linear combination of positive roots with non-negative real coefficients.

\item A dominant weight $\lambda \neq 0$ is \textbf{minuscule} if, for every positive root $\alpha$ of $G$, we have $(\lambda, \alpha^\vee) \in \{0, 1\}$.

\item Every algebraic representation $V$ of $G$ can be restricted to a representation of $T$. The action of $T$ on $V$ is diagonalisable, so we have
\begin{equation}\label{eq: weight spaces}
V \cong \bigoplus_{\lambda \in \Lambda} V_\lambda \quad \text{where} \quad V_\lambda = \{v \in V : t \cdot v = \lambda(t) v \quad \forall t \in T(\overline{\mathbb{F}}_\ell) \}.
\end{equation}

If $V_\lambda \neq \{0\}$, we say that $\lambda$ is a \textbf{weight of the representation $V$}.

\end{itemize}
\end{definition}

\begin{remark}\label{rmk: eigenvalues of semisimple elements}
Let $\rho : G \to \operatorname{GL}_V$ be a representation of $G$ and $s \in G(\overline{\mathbb{F}}_\ell)$ be a semisimple element. It is well known that $s$ is conjugate to an element $t$ in $T(\overline{\mathbb{F}}_\ell)$: indeed, $s$ is contained in some maximal torus, and all maximal tori are conjugate in $G_{\overline{\mathbb{F}}_\ell}$, see \cite[Section 21.3, Corollary A]{MR396773}. The eigenvalues of $\rho(s)$ coincide with those of $\rho(t)$, which by \Cref{eq: weight spaces} are given by $\lambda(t)$ for $\lambda$ weight of $V$ (the eigenvalue $\lambda(t)$ is counted with multiplicity $\dim V_\lambda$).
\end{remark}

Our interest in dominant weights arises from their use in classifying irreducible representations. The key statement we need is the following:

\begin{theorem}[{Theorem of the highest weight, \cite[Theorem on p.~190]{MR396773}}]
   There is a bijection between finite dimensional irreducible $\overline{\mathbb{F}}_\ell[G]$-modules $V$ and dominant weights $\lambda$, sending $V$ to its unique dominant weight that is maximal with respect to the partial ordering on $\Lambda^+$.
\end{theorem}
We will henceforth write $L(\lambda)$ for the irreducible module corresponding to $\lambda$.

\subsection{Representation theory of finite simple groups of Lie type}\label{sec_RepTheoryFiniteLieGroups}
Our main reference for this section is \cite{MR1901354}; further information can be found in \cite{MR794307}, Chapter 1 (especially sections 1.17-1.19). 

\begin{definition}\label{def: Frobenius maps}
Let $e$ be a positive integer and $q=\ell^e$. The $q$-Frobenius map of $\operatorname{GL}_n(\overline{\mathbb{F}}_\ell)$, denoted by $F_q$, is the automorphism of $\operatorname{GL}_n(\overline{\mathbb{F}}_\ell)$ that raises all coefficients of a matrix to the $q$-th power.
The \textbf{standard Frobenius map} (relative to $q$) is a group morphism $F:G(\overline{\mathbb{F}}_\ell) \to G(\overline{\mathbb{F}}_\ell)$ such that $i \circ F=F_q \circ i$, for some embedding $i:G(\overline{\mathbb{F}}_\ell) \hookrightarrow \operatorname{GL}_n(\overline{\mathbb{F}}_\ell)$ induced by an algebraic map $G \hookrightarrow \operatorname{GL}_{n, \overline{\mathbb{F}}_\ell}$. It depends only on $q$ \cite[§1.3]{MR2199819}, so we will also denote it by $F_q$ by abuse of notation.
A homomorphism $G(\overline{\mathbb{F}}_\ell) \to G(\overline{\mathbb{F}}_\ell)$ is a \textbf{Frobenius endomorphism} if some power of it is the standard Frobenius map relative to some power of $\ell$.
\end{definition}

\begin{notation}\label{rmk_InvariantQ}
Let $\tilde{G}$ be a finite simple group of Lie type in characteristic $\ell \geq 5$.
To such a group $\tilde{G}$ we associate:
\begin{itemize}
    \item a connected, reductive, simple, and simply connected algebraic group $G/\overline{\mathbb{F}}_\ell$;
    \item a Frobenius endomorphism $F$ of $G$ such that $\tilde{G} \cong \Gq/Z$, where we write $\Gq=\left\{g\in G(\overline{\mathbb{F}}_\ell) \bigm\vert F(g)=g \right\}$ for the $F$-fixed points and $Z=Z(\Gq)$ for its centre;
    \item a positive integer invariant $q=\ell^e$.
\end{itemize}

We will use these associated properties of $\tilde{G}$ without further reference.

\end{notation}

\begin{remark}
The integer $q$ arises from the choice of automorphism of the Dynkin diagram that determines $F$; one can show that it is a power of $\ell$ in our setup (see for example \cite[§1.3]{MR2199819}).
\end{remark}

Our interest in this data comes from the fact that projective representations of $\tilde{G}$ in characteristic $\ell$ correspond bijectively to linear representations of $\Gq$ in characteristic $\ell$ (\cite[p.~77, item (x)]{MR0466335}), which in turn can be constructed by restricting algebraic representations of the algebraic group $G$ to $\Gq$, as we now describe. 

\begin{theorem}[{Steinberg \cite[Theorem 43 in §13]{MR0155937}, \cite[Proposition 3.10 on p.~219]{MR2015057}}]\label{thm_Steinberg} Let $G$, $\Gq$ and $q$ be as above (with the restriction $\ell \geq 5$) and let $\omega_i$, $1 \leq i \leq r$, be the fundamental weights of $G$. Let 
\[
\Lambda_q=\left\{ a_1\omega_1+\cdots+a_r \omega_r \bigm\vert 0 \leq a_i \leq q-1 \text{ for } 1 \leq i \leq r \right\}
\]
be the set of $q$-restricted weights.
The restrictions of the $G$-modules $L(\lambda)$ with $\lambda \in \Lambda_q$ to $\Gq$ form a set of pairwise inequivalent representatives of all equivalence classes of irreducible $\overline{\mathbb{F}}_\ell[\Gq]$-modules.
\end{theorem}

The following theorem illustrates the importance of $\ell$-restricted weights: 
\begin{theorem}[{Steinberg's twisted tensor product theorem \cite[Theorem 41 in §11]{MR0155937}, \cite[Corollary on p.~224]{MR2015057}}]\label{thm_STTPT}
If $M$ is a $G$-module corresponding to the representation $\rho$, let $M^{(i)}$ be the module corresponding to the representation $\rho \circ F_\ell^i$ (see \Cref{def: Frobenius maps}).
If $\lambda_0,\ldots,\lambda_m$ are $\ell$-restricted weights, then
\[
L(\lambda_0+\ell\lambda_1+\cdots+\ell^m\lambda_m) \cong L(\lambda_0) \otimes L(\lambda_1)^{(1)} \otimes \cdots \otimes L(\lambda_m)^{(m)}.
\]
\end{theorem}

We will be interested in the weights of $L(\lambda)$. 
Recall that the Weyl group $W(G)$ of $G$ is the subgroup of $\operatorname{GL}(\Lambda \otimes \mathbb{R})$ generated by the reflections along the simple roots.

\begin{theorem}\label{thm_Premet}
Let $G$ be a simple, simply connected algebraic group in characteristic $\ell \geq 5$. Let $\lambda$ be an $\ell$-restricted weight. 
\begin{enumerate}
    \item The set of weights of the irreducible $G$-module $L(\lambda)$ is the union of the $W(G)$-orbits of dominant weights $\mu$ that satisfy $\mu \preceq \lambda$.
    \item Let $\alpha$ be a positive root and $k=(\lambda, \alpha^\vee) \in \mathbb{Z}_{\geq0}$. The weights $\lambda, \lambda-\alpha, \ldots, \lambda-k \alpha$ all appear in $L(\lambda)$.
    \end{enumerate}
\end{theorem}
\begin{proof}
By a result of Premet \cite{MR905003}, \cite[Theorem on p.~23]{MR2199819}, under our assumption $\ell \geq 5$ the weights of $L(\lambda)$ coincide with the weights of the Weyl module $V(\lambda)$ (possibly with different multiplicities). That the weights of $V(\lambda)$ are those described in the statement is well known, see for example \cite[Equation (1) on p.~22]{MR2199819}. For part (2) see \cite[§21.3]{MR499562} (the statement there is given for $V(\lambda)$, but $L(\lambda)$ has the same weights as $V(\lambda)$).

\end{proof}

\Cref{thm_Steinberg,thm_STTPT} are all we need to describe representations over $\overline{\mathbb{F}}_\ell$. However, we will also need to know when an $\overline{\mathbb{F}}_\ell$-representation $\rho$ can be defined over $\mathbb{F}_\ell$. 
Observe that every irreducible $\overline{\mathbb{F}}_\ell[\Gq]$-module $M$ corresponds to an irreducible $\overline{\mathbb{F}}_\ell[G]$-module $M \cong L(\lambda)$ by \Cref{thm_Steinberg} for some $q=\ell^e$-restricted $\lambda$; \Cref{thm_STTPT} enables us to decompose this as $M=\bigotimes_{i=0}^{e-1} M_i^{(i)}$ where the $M_i$ correpsond to $\ell$-restricted weights.

\begin{proposition}\label{prop_Descent}
With the above notation, the $\Gq$-module $M$ can be defined over $\mathbb{F}_{\ell}$ if and only if $M_i \cong M_j$ for all $i,j$.
\end{proposition}
\begin{proof}
This is the case $f=1$ of \cite[Proposition 5.4.6]{MR1057341} (using \cite[Remark 5.4.7]{MR1057341} if $\Gq$ is of type ${}^3D_4$). See also the proof of \cite[Theorem 5.1.13]{MR3098485}.
\end{proof}

\subsection{Restrictions on the Lie type}

We also need some information about the duality properties of $L(\lambda)$. We have the following result of Steinberg (\cite[Lemmas 78 and 79]{MR0466335}, cf.~also \cite[Appendix A.3]{MR1901354}):
\begin{theorem}\label{thm_Duality}
Assume $\ell \geq 5$. Let $\lambda=\sum_{i=1}^r a_i\omega_i$ be a $q$-restricted weight.
\begin{itemize}
\item If $G$ is of type $A_r$ with $r \geq 2$, or $D_r$ with odd $r$, or $E_6$, the representations $L\left( \sum_{i=1}^r a_i\omega_i\right)$ and $L\left(\sum_{i=1}^r a_{\tau(i)}\omega_i \right)$, where the permutation $\tau$ is given by the automorphism of order two of the Dynkin diagram, are dual to each other. For any other $G$, all representations $L(\lambda)$ are self-dual.
\item There is an element $h \in Z(G(\overline{\mathbb{F}}_\ell))$ with $h^2=\operatorname{id}$ such that every self-dual module $L(\lambda)$ is symplectic if and only if $h$ acts nontrivially on $L(\lambda)$.
\end{itemize}
\end{theorem}

It is then relatively easy to work out which representations $L(\lambda)$ are symplectic.
\begin{remark}
\Cref{thm_Duality} seems to be misquoted in \cite{MR1901354}, and as a consequence, the algorithm described in that paper to decide whether $L(\lambda)$ is symplectic or orthogonal does not give correct results (for example, it implies the existence of symplectic representations of $\operatorname{Spin}(7,\mathbb{F}_p)$ for all sufficiently large primes $p$ \cite[Appendix A.4]{MR1901354}), which is not the case).
\end{remark}

The following result follows from \Cref{thm_Duality} (see also the proof of \cite[Proposition 22]{MR3376151} and \cite[Exercise 4.3.13]{MR1064110}). The numbering of the simple roots is that of \cite{MR1901354}.

\begin{corollary}\label{cor_Symplectic}
Assume $\ell \geq 5$. In the situation of \Cref{thm_Duality}, if the representation $L(\lambda)$ of the finite group of Lie type $\Gq$ is symplectic, one of the following holds:

\begin{itemize}
 \item $G$ is of type $A_r$, $r \equiv 1 \bmod 4$, $a_i=a_{r+1-i}$ for $i=1,\ldots,r$, and $a_{(r+1)/2}$ is odd; or
 \item $G$ is of type $B_r$, $r \equiv 1,2 \bmod 4$, and $a_1$ is odd; or
 \item $G$ is of type $C_r$, and $\begin{cases}
a_1+a_3+\ldots+a_r \,\, \text{ is odd} & \quad \text{if $r$ is odd;}  \\
a_2+a_4+\ldots+a_r \,\, \text{ is odd} & \quad \text{if $r$ is even; or}
 \end{cases}$

\item $G$ is of type $D_r$, $r \equiv 2 \bmod 4$, and $a_{1}+a_2$ is odd; or
 \item $G$ is of type $E_7$ and $a_2+a_4+a_7$ is odd.
 \end{itemize}
\end{corollary}

\begin{corollary}\label{cor_NoPSL2q}
Let $q=\ell^e$ be the invariant attached to $\Gq$, and let $M$ be an absolutely irreducible, symplectic $\mathbb{F}_\ell[\Gq]$-module. The Lie type of $G$ is listed in \Cref{cor_Symplectic} and one of the following holds:
\begin{enumerate}
\item $e=1$, that is, $q=\ell$, and we have $M\cong L(\lambda)$ for some $\ell$-restricted weight $\lambda$; or
\item $e>1$ is odd, $\dim_{\mathbb{F}_\ell} M$ is a perfect $e$-th power, $M \otimes_{\mathbb{F}_\ell} \overline{\mathbb{F}}_\ell$ is a tensor-decomposable $\Gq$-module and the rank of $G$ is at most $(\dim_{\mathbb{F}_\ell} M)^{1/e}-1$.
\end{enumerate}
\end{corollary}

\begin{proof}
Let $\overline{M}:=M \otimes \overline{\mathbb{F}}_\ell$ and $\lambda$ be the associated $q$-restricted weight. Writing $\lambda=\sum_{i=0}^{e-1} \ell^i \lambda_i$, where each $\lambda_i$ is $\ell$-restricted, we have $\overline{M} \cong \bigotimes_{i=0}^{e-1} L(\lambda_i)^{(i)}$ by \Cref{thm_STTPT}. Since $\overline{M}$ can be defined over $\mathbb{F}_\ell$, \Cref{prop_Descent} gives $L(\lambda_i) \cong L(\lambda_j)$ for all $i,j$ hence $\dim (L_\lambda)=(\dim L(\lambda_0))^e$. If $e=1$ we are in case (1) and are done. 

If instead $e>1$ (in which case $\overline{M}$ is tensor-decomposable), then we need to show that $e$ is odd. Suppose by contradiction that $e$ is even. As noted in the discussion following \cite[Proposition 5.1.12]{MR3098485}, the fact that $\overline{M}$ is self-dual implies that each $L(\lambda_i)^{(i)}$ is self-dual, hence in particular $L(\lambda_0)$ is self-dual. By \Cref{thm_Duality} there is an element $h$ in the center of $\Gq$ that acts nontrivially on $L(\lambda_0)$ if and only if the latter is symplectic. Since $h$ is central and $L(\lambda_0)$ is absolutely irreducible, $h$ acts on $L(\lambda_0)$ as a scalar $\mu$; furthermore, since the order of $h$ is at most 2, we have $\mu=\pm 1$. It follows that $h$ acts on $\overline{M} \cong \otimes_{i=0}^{e-1} L(\lambda_0)^{(i)}$ as $\mu^e$, so $\overline{M}$ is symplectic if and only if $\mu=-1$ and $e$ is odd. The simple group $G$ admits a map with finite kernel to $\operatorname{GL}_{L(\lambda_0)} \otimes \cdots \otimes \operatorname{GL}_{L(\lambda_0)^{(e-1)}}$, and therefore, by semisimplicity, to $\operatorname{SL}_{L(\lambda_0)} \otimes \cdots \otimes \operatorname{SL}_{L(\lambda_0)^{(e-1)}}$. Since $\operatorname{SL}_n$ has rank $n-1$, and we have a map with trivial kernel to $\operatorname{SL}_{L(\lambda_0)}$, we get that $r \leq (\dim_{\mathbb{F}_\ell} M)^{1/e}-1$.
\end{proof}

We now concentrate on the abelian variety application and provide bounds on $\ell$ such that $G_\ell$ is not contained in a maximal subgroup of class $\mathcal{S}$ when the the invariant $q$ that we associate is prime. To this end, we set some notation for the remainder of this section. 

\begin{notation}\label{notation: Lie}
Let $K$ be a number field and consider an abelian variety $A/K$ such that $\operatorname{End}_{\overline{K}}(A)=\mathbb{Z}$. Let $H$ be a maximal subgroup of $\operatorname{GSp}_{2g}(\mathbb{F}_\ell)$ of class $\mathcal{S}$ such that $S := \operatorname{soc} (\mathbb{P}H)$ is a finite simple group of Lie type in characteristic $\ell$. Write $S=\Gq/Z$ and $q=\ell^e$ and let $r$ be the rank of $G$. By definition of class $\mathcal{S}$, the linear representation $A[\ell]$ of $G^F$ lifting the projective representation $\mathbb{P}A[\ell]$ of $G^F/Z$ is absolutely irreducible. Let $\lambda$ be the $q$-restricted weight such that $A[\ell] \otimes_{\mathbb{F}_\ell} \overline{\mathbb{F}}_\ell \cong L(\lambda)$ (see \Cref{thm_Steinberg}).    
\end{notation}

We also require some information about the outer automorphisms of simple groups of Lie type.

\begin{lemma}\label{lemma_IndexSocle}
Using \Cref{notation: Lie}, let $q=\ell^e$ where $\ell \geq 5$ is prime. The exponent of $\operatorname{Out}(S)$ divides $4e(r+1)$.
\end{lemma}

\begin{proof}
The Lie type of $\Gq$ is listed in \Cref{cor_Symplectic}. For each possible Lie type we read from \cite[§1.7.2]{MR3098485} (and \cite[p.~xvi, Table 5]{MR827219} for the case of $E_7(\ell^e)$) the structure of $\operatorname{Out}(S)$, cf.~\Cref{table: outer automorphism groups}. In every case, the exponent of $\operatorname{Out}(S)$ divides $4e(r+1)$. 
\end{proof}

\begin{table}[h!]
\begin{center}
\begin{tabular}{|c|c|c|c|}
\hline Family & Conditions & $\begin{matrix} \text{Simple group} \\ \text{(notation as in \cite{MR3098485})} \end{matrix}$ & $\operatorname{Out}(S)$\\[2pt] \hline
$A_1$ & & $L_2(\ell^e)$ & $\mathbb{Z}/2\mathbb{Z} \times \mathbb{Z}/e\mathbb{Z}$ \\
$A_r$ & $r \geq 2$ & $L_{r+1}(\ell^e)$ & $\mathbb{Z}/(\ell^e-1,r+1)\mathbb{Z} \rtimes \mathbb{Z}/2e\mathbb{Z}$ \\ \hline

$^2A_r$ & $r \geq 2$ & $U_{r+1}(\ell^e)$ & $\mathbb{Z}/(\ell^e+1,r+1)\mathbb{Z} \rtimes \mathbb{Z}/2e\mathbb{Z}$ \\ \hline

$B_r$ & & $O^{\circ}_{2r+1}(\ell^e)$ & $\mathbb{Z}/2\mathbb{Z} \times \mathbb{Z}/e\mathbb{Z}$ \\ \hline
$C_r$ & & $S_{2r}(\ell^e)$ & $\mathbb{Z}/2\mathbb{Z} \times \mathbb{Z}/e\mathbb{Z}$ \\ \hline
$D_r$ & $r \equiv 2 \pmod{4}, \; r\geq 6$ & $O^{+}_{2r}(\ell^e)$ & $D_4 \times \mathbb{Z}/e\mathbb{Z}$ \\ 
\hline

$^2D_{r}$ & $r \equiv 2 \pmod{4}, \; r\geq 6$ & $O^{-}_{2r}(\ell^e)$ & $\mathbb{Z}/2\mathbb{Z} \times \mathbb{Z}/2e\mathbb{Z}$ \\ 

\hline

$E_7$ & & $E_7(\ell^e)$ & $\mathbb{Z}/2\mathbb{Z} \times \mathbb{Z}/e\mathbb{Z}$ \\ \hline
\end{tabular}
\end{center}
\caption{Outer automorphism groups ($\ell$ odd); we only include the Lie types of \Cref{cor_Symplectic}.}\label{table: outer automorphism groups}
\end{table}

If $e>1$, \Cref{cor_NoPSL2q} says that $\op{soc} (\mathbb{P}A[\ell])$ is tensor-decomposable and hence we may now eliminate this case using the tools from \S\ref{sect_TensorProductsI}. 

\begin{proposition}\label{prop_e>1}
Let $\ell$ and $v$ be as in \Cref{assump_degreePolarization}. With \Cref{notation: Lie}, assume that $q\neq \ell$ (so $e \neq 1$).
If $\ell > (2q_v^{4e(r+1)})^{[F(v):\mathbb{Q}]}$, then $G_\ell$ is not contained in $H$. Moreover, $e(r+1) \leq \log_2(2g)\sqrt[3]{2g}$.
\end{proposition}

\begin{proof}
Let $N$ be the exponent of $\operatorname{Out}(S)$. Since $\mathbb{P}H$ is almost simple, we apply \Cref{lemma_socGIsPerfect} to see that $h^N \in S$ for all $h \in \mathbb{P}H$. Moreover, we know that $S$ is tensor-decomposable by \Cref{cor_NoPSL2q}(2), so we may apply \Cref{prop_NoELargerThan1} with $t=e$ and $J=\pi^{-1}(S)$ where $\pi: H \rightarrow \mathbb{P}H$ is the canonical projection. It then follows that $\ell \leq (2q_v^N)^{[F(v):\mathbb{Q}]}$, where $N \leq 4e(r+1)$ by \Cref{lemma_IndexSocle}. Now \Cref{cor_NoPSL2q}(2) implies that $e$ is odd and hence $e\geq 3$ so  $r \leq \sqrt[3]{2g} -1$ and $3\leq e \leq \log_2(2g)$ hence $e(r+1) \leq \log_2(2g)\sqrt[3]{2g}$.
\end{proof}

Having treated the $e>1$ case, we henceforth restrict to the $e=1$ setting for the remainder of this section.

\subsection{The non-minuscule case}

We first treat the case where $\lambda$ is not a minuscule weight
\Cref{def_weights}. In this situation, our bound will depend on the rank of the associated Lie group $G$ which we first bound.

\begin{proposition}\label{prop: rank bound}
Suppose that the invariant $q$ attached to $G$ is equal to $\ell$ and let $L(\lambda)$ be a symplectic representation of the finite group of Lie type $\Gq$. Then either: $G$ is of type $C_r$ and $L(\lambda)$ is the defining representation of $\operatorname{Sp}_{2r, \overline{\mathbb{F}}_\ell}$; or $r \leq 2\sqrt[3]{\dim L(\lambda)}$.

\end{proposition}

\begin{proof}
For $r > 11$, this follows from \cite[Theorem 5.1]{MR1901354}: \cite[Table 2]{MR1901354} lists no symplectic representation for types $A_r, B_r$ and $D_r$, and only $\omega_r$ for type $C_r$. It follows that (under the assumption $r>11$ and with the exception of the standard representation of $C_r$) we have $\dim L(\lambda) \geq r^3/8$. For $r \leq 11$, the result follows from a tedious but straightforward examination of the tables in Appendix A.6 to Appendix A.53 of \cite{MR1901354}. To simplify the reading of these tables, notice that (when consulting the table relative to a certain type of rank $r$) we simply have to check that all the representations listed with dimension less than $r^3/8$ do not appear in \Cref{cor_Symplectic}; aside from the standard representation for $C_r$, we find that there are no symplectic ones below this bound.
\end{proof}

\begin{lemma}\label{lemma_eqs}
Let $\ell \geq 5$ be a prime number, $\lambda \neq 0$ an $\ell$-restricted weight of $G$ which is not minuscule, and $\rho:G \to \operatorname{GL}_{L(\lambda)}$ the corresponding irreducible representation. For every semisimple $x \in G(\overline{\mathbb{F}}_\ell)$ there are three eigenvalues $\lambda_1,\lambda_2,\lambda_3$ of $\rho(x)$ that satisfy $\lambda_2^2=\lambda_1\lambda_3$.
\end{lemma}

\begin{proof}
By assumption $\lambda$ is not minuscule, hence there exists a positive root $\alpha$ such that $(\lambda, \alpha^\vee) \geq 2$. By \Cref{thm_Premet}(2), the three weights $w_1=\lambda, w_2=\lambda-\alpha, w_3=\lambda-2\alpha$ all appear in $L(\lambda)$ and note that $2w_2=w_1+w_3$. As $x$ is semisimple, it is conjugate to an element $x'$ in the maximal torus $T$. By \Cref{rmk: eigenvalues of semisimple elements}, the eigenvalues of $\rho(x)$ are the values of the weights of $L(\lambda)$ at $x'$; in particular it possesses the three eigenvalues $w_1(x'), w_2(x'), w_3(x')$, which satisfy $w_2(x')^2=w_1(x')w_3(x')$.
\end{proof}

\begin{proposition}\label{prop_ConclusionLie}
Let $\ell$ and $v$ be as in \Cref{assump_degreePolarization}. With \Cref{notation: Lie}, assume that $e=1$ and $\lambda$ is not a minuscule weight.
If $\ell > (2q_v^{4(r+1)})^{[F(v):\mathbb{Q}]}$, then $G_\ell$ is not contained in $H$.
\end{proposition}

\begin{proof}
Suppose by contradiction that $G_\ell$ is contained in $H$. By \Cref{cor_NoPSL2q}, since $e=1$ the weight $\lambda$ is $\ell$-restricted.
The following commutative diagram summarises the situation:
\begin{center}\makebox{
\xymatrix{
\operatorname{Gal}\left(\overline{K}/K \right) \ar@/_/[dr]_{\pi \, \circ \rho_\ell} \ar[r]_{\rho_\ell} & \operatorname{Aut} A[\ell] \ar[d]^\pi
&  \ar@/^/[dl]^{\pi \circ \rho} \Gq \ar[l]^{\hspace{10pt} \rho} \\
& \mathbb{P} \operatorname{Aut} A[\ell] }}
\end{center}
where $\operatorname{soc} \, (\mathbb{P}H) =  \operatorname{Im} (\pi \circ \rho)$.
Since by assumption $G_\ell \subseteq H$ and $e=1$, it follows from \Cref{lemma_IndexSocle} that for every $y\in \operatorname{Gal}\left(\overline{K}/K \right)$ there exist $\mu \in \mathbb{F}_\ell^\times$ and $x \in \Gq \subseteq G\left(\overline{\mathbb{F}}_\ell\right)$ such that
\[
\rho_\ell(y)^\MaxOrderOut = \mu \rho(x),
\]
and therefore the eigenvalues of $\rho_\ell(y)^\MaxOrderOut$ are given by $\{\mu \lambda_1, \ldots, \mu \lambda_{2g} \}$, where $\lambda_1,\ldots,\lambda_{2g}$ are the eigenvalues of $\rho(x)$. If $\rho(x)$ is semisimple, since $\lambda$ is not minuscule,
\Cref{lemma_eqs} shows that (up to renumbering the $\lambda_i$) we have $\lambda_2^2=\lambda_1\lambda_3$, and therefore $(\mu \lambda_2)^{\TwiceMaxOrderOut}=(\mu \lambda_1)^{\MaxOrderOut}(\mu \lambda_3)^{\MaxOrderOut}$.
We have thus proved that, if $y \in \operatorname{Gal}\left(\overline{K}/K \right)$ is such that $\rho_\ell(y)$ is semisimple,
then it has three eigenvalues $\mu_1,\mu_2, \mu_3 \in \overline{\mathbb{F}}_\ell^\times$ that satisfy
\begin{equation*}\label{eq_Power8}
\mu_2^{\TwiceMaxOrderOut}=\mu_1^\MaxOrderOut\cdot \mu_3^\MaxOrderOut \text{ in } \overline{\mathbb{F}}_\ell.
\end{equation*}
For $y=\operatorname{Fr}_v$, $\rho_{\ell}(y)$ is semisimple by \Cref{lemma: semisimple Frob}, so \Cref{prop_MultRelBoundsEll}(1) concludes the proof.

\end{proof}

\subsection{The minuscule case, $e=1$}

We now consider maximal subgroups of class $\mathcal{S}$ acting through representations defined by minuscule weights. In the rest of this section, we continue using \Cref{notation: Lie} and assume that $e=1$ and $\lambda$ is a minuscule weight. 
We begin by describing a sufficient criterion that
ensures that $G_\ell$ is \textit{not} contained in a group $H$ as in \Cref{notation: Lie}, under these assumptions. 
Given the very limited number of minuscule weights, it will then be easy to check that this criterion applies to every minuscule weight giving rise to a symplectic representation.

\begin{lemma}\label{lemma_LieCriterion}
Let $\ell$ and $v$ be as in \Cref{assump_degreePolarization}. Using \Cref{notation: Lie}, assume that $G_\ell$ is contained in $H$ and $q=\ell$.
Let $\lambda_1,\ldots,\lambda_{2g}$ be the weights of the representation $L(\lambda)$. Suppose we can find an odd integer $n$ such that (up to renumbering) the weights $\lambda_i$ satisfy
\[
\lambda_1+\ldots+\lambda_n = \lambda_{n+1}+\ldots+\lambda_{2n}.
\]
Then the inequality $\ell \leq (2q_v^{nN/2})^{[F(v):\mathbb{Q}]}$ holds, where $N$ is the exponent of $\operatorname{Out}(S)$.
\end{lemma}
\begin{proof}

Let $s \in \operatorname{GSp}_{2g}(\mathbb{F}_\ell)$ be semisimple. Suppose that the projective image $[s]$ of $s$ lands in $S$. As in the proof of \Cref{lemma_eqs}, the operator $s$ admits $2n$ eigenvalues $\overline{\mu}_1,\ldots,\overline{\mu}_{2n}$ that satisfy $\overline{\mu}_1 \cdots \overline{\mu}_n = \overline{\mu}_{n+1} \cdots \overline{\mu}_{2n}$. By definition of $N$, the projective image $[h]$ of $h \in G_\ell$ satisfies $[h]^N \in S$, so $h^N$ has $2n$ eigenvalues that satisfy the previous equation. Since the eigenvalues of $s$ are the $N$-th powers of the eigenvalues of $h$, we deduce that every semisimple operator $h \in G_\ell$ has $2n$ eigenvalues $\overline{\mu}_1,\ldots,\overline{\mu}_{2n}$ that satisfy
\begin{equation}\label{eq_muToTheN}
\overline{\mu}_1^N \cdots \overline{\mu}_n^N =\overline{\mu}_{n+1}^N \cdots \overline{\mu}_{2n}^N \text{ in } \overline{\mathbb{F}}_\ell.
\end{equation}
Specialise now this discussion to $h=\rho_\ell(\operatorname{Fr}_v)$, which is semisimple by \Cref{lemma: semisimple Frob}. For this choice of $h$ there are $2n$ eigenvalues of $\operatorname{Fr}_v$, call them $\mu_1,\ldots,\mu_{2n}$, whose reductions in $\overline{\mathbb{F}}_\ell$ satisfy Equation \eqref{eq_muToTheN}. By \Cref{prop_MultRelBoundsEll}(3) we obtain $\ell \leq (2q_v^{nN/2})^{[F(v):\mathbb{Q}]}$.
\end{proof}

We will show the following.
\begin{proposition}\label{prop_NoMinusc}
Let $\ell$ and $v$ be as in \Cref{assump_degreePolarization}. Using \Cref{notation: Lie}, assume that $q=\ell$ and $\lambda$ is a minuscule weight. If $\ell > (2q_v^{6(r+1)})^{[F(v):\mathbb{Q}]}$, the group $G_\ell$ is not contained in $H$. Moreover $r$ satisfies $6(r+1) \leq 4(\log_2(2g)\sqrt[3]{2g}+1)$.

\end{proposition}
\begin{proof}
The list of symplectic minuscule weights is given below in \Cref{prop: symplectic minuscule weights}. In the next subsections, for each weight $\lambda$ described in \Cref{prop: symplectic minuscule weights} we give an upper bound on $r$ using that $\dim L(\lambda)=\dim A[\ell]=2g$ and list six distinct weights $\lambda_1, \ldots, \lambda_6$ of $L(\lambda)$ with $\lambda_1+\lambda_2+\lambda_3=\lambda_4+\lambda_5+\lambda_6$.  The bound on $\ell$ follows from \Cref{lemma_LieCriterion}, where we take $N=4(r+1)$ by \Cref{lemma_IndexSocle} in \Cref{lemma_LieCriterion}.

It remains to prove the bound on $r$. Let $B(g)=4(\log_2(2g)\sqrt[3]{2g}+1)$. Note that all bounds on $r$ given are at most $2\log_2(g)+1$ except in the $B_r$ case; however we can assume that this case only occurs with $r \geq 5$ (see below), at which point our bound also holds. Observe that the inequality $12(\log_2(g)+1) \leq B(g)$ holds for $g \geq 12$. For smaller $g$, the only case that occurs where this fails is $A_5$ where one checks the rank bound directly.
\end{proof}

A list of all minuscule weights is given in \cite[p.~132]{Bourbaki79}, while Tables 1 and 2 of \textit{op.~cit.}~show which minuscule modules are symplectic and give their dimensions. 
\begin{proposition}\label{prop: symplectic minuscule weights}
The complete list of symplectic, minuscule representations is as follows:

\begin{center}
    \begin{tabular}{c|c|c|c}
     Type & Condition & Weight & Dimension \\ \hline
     $A_r$ & $r \equiv 1 \pmod 4$ & $\varpi_{(r+1)/2}$ & $\binom{r+1}{\frac{r+1}{2}}$ \\
     $B_r$ & $r \equiv 1,2 \pmod 4, r \geq 2$ & $\varpi_r$ & $2^r$ \\
     $C_r$ & $r \geq 2$ & $\varpi_1$ & $2r$ \\
     $D_r$ & $r \equiv 2 \pmod 4, r \geq 3$ & $\varpi_r, \varpi_{r-1}$ & $2^{r-1}$ \\
     $E_7$ & & $\varpi_7$ & $56$
\end{tabular}
\end{center}

\end{proposition}
\begin{remark}
    Note that the numbering of the fundamental weights in \cite{Bourbaki79} is different from that in \cite{MR1901354}, which we used for \Cref{cor_Symplectic}.
\end{remark}

Using \Cref{notation: Lie}, we realise the root system of $G$ in the Euclidean vector space $\mathbb{R}^n$ as described in \cite[Planches I-IX]{MR1890629}. We denote by $\varepsilon_1,\ldots,\varepsilon_n$ the canonical basis of $\mathbb{R}^n$ and express the simple roots of $G$ and its fundamental weights in terms of $\varepsilon_1,\ldots,\varepsilon_n$. If $\lambda$ is a dominant weight that is also minuscule, the set of weights appearing in $L(\lambda)$ is the orbit of $\lambda$ under the Weyl group $W(G)$, see \Cref{thm_Premet}(1) and recall that $\lambda$ is minimal for the dominance order (this fact implies that $L(\lambda)$ has the same dimension in positive characteristic as it does in characteristic $0$). 
A description of the action of $W(G)$ on $\varepsilon_1,\ldots,\varepsilon_n$ can be read off \cite[Planches I-IX]{MR1890629}. We will make use of this information without further reference to \cite{MR1890629}.

\subsubsection{Lie type $A_r$, $r \equiv 1 \pmod 4$}
The only weight to consider is $
\varpi_{(r+1)/2} = \frac{1}{2}\sum_{i=1}^{(r+1)/2} \varepsilon_i - \frac{1}{2} \sum_{i=(r+3)/2}^{r+1} \varepsilon_i.$

When $r=1$, the module $L(\varpi_{(r+1)/2})$ is the standard representation of $\operatorname{SL}_2$, of dimension 2, which can only happen for $g=1$, a case we exclude. We can then assume $r \geq 5$.
The Weyl group of $A_r$ is $\op{Sym}_{r+1}$, acting naturally on the $\varepsilon_i$ by permutation. The weights appearing in $L(\varpi_{(r+1)/2})$ are then those of the form $\sum_{i=1}^{r+1} a_i \varepsilon_i$, where $a_i=\pm \tfrac{1}{2}$ with precisely half of them positive. We set $ \omega=\sum_{i=7}^{r+1} (-1)^i\varepsilon_i$ (where $\omega=0$ if $r=5$) and define $\lambda_1, \ldots, \lambda_6$ by
\[
\begin{aligned}
2\lambda_1 & = -\varepsilon_1 + \varepsilon_2 + \varepsilon_3 + \varepsilon_4 - \varepsilon_5 -\varepsilon_6 + \omega, & \quad 2\lambda_4 & = +\varepsilon_1 + \varepsilon_2 - \varepsilon_3 + \varepsilon_4 - \varepsilon_5 - \varepsilon_6 + \omega, \\
2\lambda_2 &  = + \varepsilon_1 - \varepsilon_2 + \varepsilon_3 - \varepsilon_4  + \varepsilon_5 -\varepsilon_6 + \omega, & 2\lambda_5 &  = - \varepsilon_1 + \varepsilon_2 + \varepsilon_3 - \varepsilon_4  + \varepsilon_5 - \varepsilon_6 + \omega, \\
2\lambda_3 &  = + \varepsilon_1 + \varepsilon_2 - \varepsilon_3 - \varepsilon_4  - \varepsilon_5 + \varepsilon_6 + \omega, &2\lambda_6 &  = + \varepsilon_1 - \varepsilon_2 + \varepsilon_3 - \varepsilon_4  -\varepsilon_5 + \varepsilon_6 + \omega.\end{aligned}
\]
From $2g=\dim A[\ell]=\dim L(\lambda)=\binom{r+1}{(r+1)/2} \geq 2^{(r+1)/2}$ we obtain $r \leq  2\log_2(g)+1$.

\subsubsection{Lie type $B_r$, $r \equiv 1,2 \pmod 4, r \geq 2$}
The only weight to consider is $\varpi_r=\frac{1}{2} \sum_{i=1}^r \varepsilon_i$. For $r=2$ this is the standard representation of a group of type $B_2=C_2$, which does not lead to a subgroup of class $\mathcal{S}$. Assume therefore $r \geq 5$ (it corresponds to the defining representation).

The Weyl group is isomorphic to $(\mathbb{Z}/2\mathbb{Z})^{r} \rtimes \op{Sym}_r$, with the factor $\op{Sym}_r$ acting on the $\varepsilon_i$ by permutation, while the nontrivial element in the $i$-th factor $\mathbb{Z}/2\mathbb{Z}$ acts on $\varepsilon_i$ by sending it to $-\varepsilon_i$.
We set $\omega=\sum_{i=5}^r \varepsilon_i$ and define
\[
\begin{aligned}
2\lambda_1 &=+\varepsilon_1 + \varepsilon_2 + \varepsilon_3 + \varepsilon_4 + \omega, \qquad & 2\lambda_4 &=-\varepsilon_1 + \varepsilon_2 + \varepsilon_3 + \varepsilon_4 + \omega, \\
2\lambda_2 &=-\varepsilon_1 - \varepsilon_2 + \varepsilon_3 + \varepsilon_4 + \omega, & 2\lambda_5 &=+\varepsilon_1 - \varepsilon_2 - \varepsilon_3 + \varepsilon_4 + \omega, \\
2\lambda_3 &=+\varepsilon_1 + \varepsilon_2 - \varepsilon_3 - \varepsilon_4 + \omega, & 2\lambda_6 &=+\varepsilon_1 + \varepsilon_2 + \varepsilon_3 - \varepsilon_4 + \omega.
\end{aligned}
\]

From $2g=\dim A[\ell]=\dim L(\lambda)=2^r$ we obtain $r = \log_2(g)+1$.

\subsubsection{Lie type $C_r$, $r \geq 2$}
The only minuscule, symplectic module for a group of type $C_r$ is the defining representation $L(\varpi_1)$, of dimension $2r$ (so $r=g$). This representation does not give rise to a maximal class $\mathcal{S}$ subgroup of $\operatorname{GSp}_{2g}(\mathbb{F}_\ell)$.

\subsubsection{Lie type $D_r$, $r \equiv 2 \pmod 4$, $r \geq 3$}
The weights to consider are $\varpi_{r-1}=\frac{1}{2}\sum_{i=1}^{r-1} \varepsilon_{i} - \frac{1}{2}\varepsilon_{r}$ and $\varpi_{r}=\frac{1}{2}\sum_{i=1}^{r} \varepsilon_{i}$. The Weyl group is isomorphic to $(\mathbb{Z}/2\mathbb{Z})^{r-1} \rtimes \op{Sym}_r$, where the factor $\op{Sym}_r$ acts by permutation on the $\varepsilon_i$ while the factor $(\mathbb{Z}/2\mathbb{Z})^{r-1} \cong \left\{ e_1,\ldots,e_r \in \{\pm 1\} \bigm\vert \prod_{i=1}^r e_r=1 \right\}$ acts by the formula $(e_1,\ldots,e_r) \cdot \sum_{i=1}^r a_i\varepsilon_i=\sum_{i=1}^r e_ia_i\varepsilon_i$. 
For the weight $\varpi_r$ we set $\omega=\sum_{i=7}^{r}\varepsilon_i$ (or $0$ if $r=6$) and take 
\[
\begin{aligned}
2\lambda_1 &= +\varepsilon_1 +\varepsilon_2 +\varepsilon_3 +\varepsilon_4 +\varepsilon_5 +\varepsilon_6 + \omega, \quad & 2\lambda_4 &=  -\varepsilon_1 -\varepsilon_2 +\varepsilon_3 +\varepsilon_4 +\varepsilon_5 +\varepsilon_6 + \omega, \\
2\lambda_2 &= -\varepsilon_1 -\varepsilon_2 -\varepsilon_3 +\varepsilon_4 -\varepsilon_5 +\varepsilon_6 + \omega, & 2\lambda_5 &=  +\varepsilon_1 +\varepsilon_2 -\varepsilon_3 -\varepsilon_4 +\varepsilon_5 +\varepsilon_6 + \omega, \\
2\lambda_3 &= +\varepsilon_1 +\varepsilon_2 +\varepsilon_3 -\varepsilon_4 +\varepsilon_5 -\varepsilon_6 + \omega, & 2\lambda_6 &=  +\varepsilon_1 +\varepsilon_2 +\varepsilon_3 +\varepsilon_4 -\varepsilon_5 -\varepsilon_6 + \omega.
\end{aligned}
\]

For $\varpi_{r-1}$ it suffices to change the sign of $\varepsilon_1$ in the weights just listed. As $2g=\dim A[\ell]=\dim L(\lambda) = 2^{r-1}$ we obtain $r = \log_2(g)+2$.

\subsubsection{Lie type $E_7$}

The only minuscule weight giving rise to a symplectic representation is $\varpi_7=\varepsilon_6 + \frac{1}{2}\left( \varepsilon_8-\varepsilon_7 \right)$. We take
\[
\begin{aligned}
2\lambda_1 &= +\varepsilon_1 +\varepsilon_2 +\varepsilon_3 +\varepsilon_4 +\varepsilon_5 +\varepsilon_6, \qquad & 2\lambda_4 &= -\varepsilon_1 -\varepsilon_2 +\varepsilon_3 +\varepsilon_4 +\varepsilon_5 +\varepsilon_6,\\
2\lambda_2 &= -\varepsilon_1 -\varepsilon_2 -\varepsilon_3 +\varepsilon_4 -\varepsilon_5 +\varepsilon_6, & 2\lambda_5 &= +\varepsilon_1 +\varepsilon_2 -\varepsilon_3 -\varepsilon_4 +\varepsilon_5 +\varepsilon_6, \\
2\lambda_3 &= +\varepsilon_1 +\varepsilon_2 +\varepsilon_3 -\varepsilon_4 +\varepsilon_5 -\varepsilon_6, & 2\lambda_6 &= +\varepsilon_1 +\varepsilon_2 +\varepsilon_3 +\varepsilon_4 -\varepsilon_5 -\varepsilon_6.
\end{aligned}
\]
The representation is 56-dimensional, so $g=28$. The rank is $r=7$.

\bigskip

We have thus exhausted the list of all the minuscule weights that give rise to symplectic representations, and the proof of \Cref{prop_NoMinusc} is complete.

\section{Proof of Main Theorem}\label{sect_MainProof}

In this section, we shall complete the proof of \Cref{thm_Main} by amalgamating our previous case analysis for each class of maximal subgroups, excluding each case by virtue of our lower bound on $\ell$. To simplify this, we first prove some inequalities for the bounds on $\ell$ that we will use, separating the proposition into two parts according to the type of bound we use.

Recall that $b(A/K;d)=\left( (7g)^{8g^2} d[K:\mathbb{Q}] \max\left(h(A), \log d[K:\mathbb{Q}],1 \right) \right)^{2g^2}$ is our explicit isogeny bound; 
one can see directly that it is an increasing function in $d$.

\begin{proposition}\label{prop_bounds1} 
Let $g \geq 2$ and define $B(g)=4(\log_2(2g)\sqrt[3]{2g} + 1)$. Recall that $\xi(t) = \max_{\sigma \in \op{Sym}_t} \operatorname{order}(\sigma)$. The following inequalities hold:
\begin{enumerate}

\item $b(A/K;(2g+5)!)^{1/(2g-1)} \leq b(A/K;2g)$; 

\item $4(2\sqrt[3]{2g} +1) \leq B(g)$;
\item $\xi(\lfloor \log_2(2g) \rfloor) \leq B(g)$.
\end{enumerate}
\end{proposition}

\begin{proof}
    These are all elementary; for (3), we use \cite[Theorem 2]{MR979940}, which gives $\xi(t) \leq \exp(1.2\sqrt{t\log t})$ for all $t \geq 3$.
\end{proof}

We now prove the main theorem of this paper, of which \Cref{thm_Main} is an immediate corollary by the discussion after \Cref{rmk: genus 2}. Throughout the proof, we will freely use \Cref{prop_bounds1} for contradictions without further reference.

\begin{theorem}[=\Cref{thm_Main}]\label{thm_MainProof}
Let $A/K$ be an abelian variety of dimension $g \geq 2$ and $G_{\ell^\infty}$ be the image of the natural representation $\rho_{\ell^\infty} : \abGal{K} \to \operatorname{Aut} T_\ell A$. Suppose that:
\begin{enumerate}
\item $\operatorname{End}_{\overline{K}}(A)=\mathbb{Z}$;
\item there exists a place $v$ of $K$, of good reduction for $A$ and with residue field of order $q_v$, such that the roots of characteristic polynomial of the Frobenius at $v$ generate a free group of rank $g+1$.

\end{enumerate}
Then the equality $G_{\ell^\infty}=\operatorname{GSp}_{2g}(\mathbb{Z}_\ell)$ holds for every prime $\ell$ unramified in $K$ and such that
\[
\ell > \max\left\{ \left(2q_v^{4(\log_2(2g)\sqrt[3]{2g} + 1)}\right)^{[F(v):\mathbb{Q}]}, \; b(A/K; 2g) \right\}.
\]
\end{theorem}

\begin{proof}
Let $\ell$ be such a prime and recall that by \Cref{cor_Conclusion}, it suffices to show that $\operatorname{Sp}_{2g}(\mathbb{F}_\ell) \subset G_\ell$; this we do by proving that $G_\ell$ is not contained in any other maximal subgroup $G$ as described in Aschbacher's result (\Cref{thm_Aschbacher}), contradicting the bound on $\ell$. If $G$ is of class $\mathcal{C}_1, \mathcal{C}_2$, $\mathcal{C}_3$, $\mathcal{C}_4$ or $\mathcal{C}_7$, then the bound follows from the following table.

\begin{center}
\begin{tabular}{c|l|l}
  \textbf{Class} & \textbf{Upper bound on $\ell$} & \textbf{Reference} \\ \hline
  $\mathcal{C}_1, \mathcal{C}_2, \mathcal{C}_3$ & $b(A/K;2g)$ & \Cref{prop_EasyCases} \\
  $\mathcal{C}_4$ & $(2q_v)^{[F(v):\mathbb{Q}]}$ & \Cref{prop:tensor products} \\
  $\mathcal{C}_7$ & $\left(2q_v^{\xi(\lfloor \log_2(2g) \rfloor}\right)^{[F(v):\mathbb{Q}]}$ & \Cref{prop: NoInducedTensors rewrite} \\
\end{tabular}
\end{center}

If $G$ is of class $\mathcal{C}_6$, then using \Cref{rmk_OrdBizarreGroup}, we are able to bound $\#\mathbb{P}G$ (and hence $\#\mathbb{P}G_\ell$) by $(2g+5)!$. However, \Cref{prop_UnconditionalConstantGroups} implies that this is only possible if $\ell \leq b(A/K;(2g+5)!)^{1/(2g-1)} \leq b(A/K;2g)$.

Lastly, we consider the case where $G$ is in the exceptional class $\mathcal{S}$. We may assume that $G$ is primitive. Since $\ell> b(A/K,J(2g))^{1/(2g-1)}$ by \Cref{prop_bounds1}(1) (where we note $J(2g) \leq (2g+5)!$ by \Cref{thm: Collins}), \Cref{prop_UnconditionalConstantGroups} tells us that $\mathbb{P}G$ is almost simple, with $S=\operatorname{soc}(\mathbb{P}G)$ being of Lie type in characteristic $\ell$. As in \Cref{sect_LieGroups}, we hence write $S=\Gq/Z$ and consider $A[\ell]$ as a representation of $\Gq$ corresponding to the representation $\mathbb{P}A[\ell]$ of $S$, and let $q=\ell^e$ be the corresponding invariant and $r$ be the rank of $S$. 

If $q \neq \ell$ (i.e.~$e\neq 1$), then we have $\ell \leq (2q_v^{4\log_2(2g)\sqrt[3]{2g}})^{[F(v):\mathbb{Q}]}$ by \Cref{prop_e>1}, so we suppose that $e=1$. Write $A[\ell] \otimes \overline{\mathbb{F}}_\ell = L(\lambda)$ by \Cref{cor_NoPSL2q}(1). If $\lambda$ is not minuscule, then $\ell \leq \left( 2q_v^{4(2\sqrt[3]{2g}+1)} \right)^{[F(v):\mathbb{Q}]}$ by \Cref{prop_ConclusionLie} and \Cref{prop: rank bound}, where we note that the defining representation of $\operatorname{Sp}_{2r,\overline{\mathbb{F}}_\ell}$ is in the minuscule case. If $\lambda$ is minuscule, then we instead apply \Cref{prop_NoMinusc} to get $\ell \leq \left( 2q_v^{4(\log_2(2g)\sqrt[3]{2g}+1)}\right)^{[F(v):\mathbb{Q}]}.$
\end{proof}

\section{Computational bounds on non-surjective primes}\label{sect: CompBounds}
We now take a more algorithmic approach in order to sharpen the bound of \Cref{thm_Main} for a given abelian variety $A$. We suppose that the image of $\rho_{A,\ell^\infty}$ is $\operatorname{GSp}_{2g}(\mathbb{Z}_\ell)$ for all $\ell$ large enough (equivalently, for some $\ell$), which in particular implies that $\operatorname{End}_{\overline{K}}(A)=\mathbb{Z}$.

\begin{remark}
\begin{enumerate}
\item We will assume that $A$ is principally polarised for simplicity; however all arguments used are valid for any prime not dividing the degree of a fixed polarisation.
\item In principle, we do not need to know that $\rho_{A,\ell^\infty}$ is surjective for sufficiently large $\ell$ to run this algorithm and the output could potentially provide a certificate of this. However, one should be wary that if it is not true then issues will arise: for example, there could be a positive density of Frobenius elements with trace 0 which would affect the utility of \Cref{lem: general phi}.
\end{enumerate}
\end{remark}

\begin{notation}
We use the following notation in this section.

\begin{tabular}{cl}
$A/K$ & a principally polarised abelian variety of dimension $g$; \\
$\mathfrak{N}_A$ & conductor of $A/K$; \\
$\ell$ & a fixed rational prime; \\
$S_\ell$ & $= \{ v \text{ place of }K : v \mid \ell\infty  \text{ or $A$ does \emph{not} have semistable reduction at } v\}$; \\
$\mathfrak{l}$ & a place of $K$ above $\ell$ of ramification degree $e(\mathfrak{l} \mid \ell)$; \\
$I_{\mathfrak{l}}$ & the inertia group at $\mathfrak{l}$; \\
$\sigma \mid_{I_{\mathfrak{l}}}$ & the restriction of a representation $\sigma$ to $I_{\mathfrak{l}}$; \\
$f_v$ & characteristic polynomial of $\rho_{\ell^\infty}(\op{Fr}_v)$ for $v \nmid \ell$.
\end{tabular}
\end{notation}

Fix a maximal subgroup $M$ of $\operatorname{GSp}_{2g}(\mathbb{F}_\ell)$ with $\det M = \mathbb{F}_\ell^\times$, and suppose that $G_\ell \subseteq M$. In the next subsections we show how to bound the prime $\ell$, depending on the type of $M$, in terms of certain well-chosen characteristic polynomials of Frobenius. In most cases, we will obtain what we call a \textit{divisibility bound}, namely, we will show that the containment $G_\ell \subseteq M$ implies $\ell \mid D_M$ for a computable, nonzero integer $D_M$ depending on $A$ and the type of $M$. For a few maximal subgroups $M$ (namely, the constant groups in class $\mathcal{S}$) we will only obtain an inequality $\ell \leq D_M$ for a suitable integer $D_M$, and we will explain how to quickly test all primes $\ell$ up to the bound $D_M$. The advantage of the approach of this section over \Cref{thm_Main} is that we get much more manageable bounds, although with a less clear dependence on invariants of $A$. 
The arguments in this section apply to abelian varieties of arbitrary dimension over arbitrary number fields, with the exception of those concerning the irreducibility of the representation $A[\ell] \otimes \overline{\mathbb{F}}_\ell$, for which we will specialise to $K=\mathbb{Q}$ and $\dim A =3$. For simplicity of exposition, however, we will often take abelian threefolds as our model case.

\subsection{Preliminaries}

For several of the Aschbacher classes, it is useful to understand the inertia action on $A[\ell]$ at every prime. The following two theorems provide this description under a semistability assumption, which is true for all but finitely many primes $v$.

\begin{proposition}[{\cite[Expos\'{e} IX, Proposition 3.5]{MR354656}}]
\label{prop:Grothendieck semistability}
Let $K_v$ be a finite extension of $\mathbb{Q}_p$. If $A/K_v$ is semistable abelian variety, then the inertia group of $\operatorname{Gal}(\overline{K}_v/K_v)$ acts unipotently on the $\ell$-adic Tate module for $A$ for some (equivalently, every) prime $\ell \neq p$. In particular, the image of inertia acting on $T_\ell A$ is a pro-$\ell$ group for all $\ell \neq p$.
\end{proposition}

The second result we use describes the inertia action at primes above $\ell$. Recall that a fundamental character of level $n$ is a certain surjective homomorphism $\psi_n: I_{\mathfrak{l}} \rightarrow \mathbb{F}_{\ell^n}^{\times}$ and all such characters of a given level are of the form $\psi_n^{\ell^m}$ for some integer $0 \leq m \leq n$. We choose a norm compatible system of fundamental characters, i.e.~such that $\psi_m = \op{N}_{\mathbb{F}_{\ell^n}/\mathbb{F}_{\ell^m}} \circ \, \psi_n = \psi_n^{\frac{\ell^n-1}{\ell^m-1}}$ whenever $m \mid n$. If $K=\mathbb{Q}$ (more generally, if $K_\mathfrak{l} \cong \mathbb{Q}_\ell$), then $\psi_1=\chi_\ell$ is the mod $\ell$ cyclotomic character.

\begin{theorem}[{\cite[Corollary 3.4.4]{MR0419467}}]\label{thm: JH fundamental chars}
Suppose $A/K$ has semistable reduction at $\mathfrak{l}\mid \ell$. Let $U$ be a simple Jordan--H\"{o}lder of the the $I_{\mathfrak{l}}$-module $A[\ell]$ with $\dim_{\mathbb{F}_\ell}U=n$. Then wild inertia acts trivially on $U$ and every eigenvalue of $U$ is of the form $\psi_n^k$ for some  $0 \leq k \leq e(\mathfrak{l} \mid \ell) \frac{\ell^n-1}{\ell-1}$.
\end{theorem}

\begin{remark}
Raynaud's theorem is usually stated for places of \textit{good} reduction. However, as it was shown in \cite[Lemma 4.9]{MR3211798}, the extension to the semistable case follows easily upon applying results of Grothendieck \cite[Exposé 9]{MR354656}.
\end{remark}

For examination of the classes $\mathcal{C}_2$ and $\mathcal{C}_3$ we will compose $\rho_\ell$ with another homomorphism. We state some simple properties first.

\begin{lemma}\label{lem: semistable phi}
Fix a subgroup $M$ of $\operatorname{GL}_{2g}(\mathbb{F}_\ell)$ and a group homomorphism $\phi : M \to H$. Suppose that $G_\ell$ is contained in $M$ and let $\varphi :  \op{Gal}(\overline{K}/K) \xrightarrow{\rho_\ell} G_\ell \subseteq M \xrightarrow{\phi} H$.

\begin{enumerate}
\item If $\ell \nmid \#H$, then the character $\varphi$ is unramified outside $S_\ell$ and $\varphi$ is (at worst) tamely ramified at all primes $\mathfrak{l} \mid \ell$.
\item Suppose $A/K$ is semistable, $\ell \nmid \#H$ and $\varphi$ is unramifed at all places above $\ell$. Then $\varphi$ is unramified at all finite places. In particular, if $K$ has narrow class number one, then $\varphi$ is trivial.
\end{enumerate}
\end{lemma}

\begin{proof}
(1): Since $\varphi$ factors through $\rho_\ell$, it is unramified outside $\mathfrak{N}_A\ell\infty$. Let $v \nmid \ell$ be a place of semistable reduction. Since the image of inertia is an $\ell$-group by \Cref{prop:Grothendieck semistability} and $\ell \nmid \#H$, $\varphi$ is unramified at $v$; similarly the image of wild inertia is trivial at places $\mathfrak{l} \mid \ell$.
(2) follows immediately from (1).
\end{proof}

\begin{remark}
\begin{enumerate}
\item When we come to use \Cref{lem: semistable phi} for classes $\mathcal{C}_2$ and $\mathcal{C}_3$, $\varphi$ being unramified at all places above $\ell$ will follow from semistability there, together with  $\ell$ being larger than an explicit bound which is independent of $\ell$; in particular this holds for all but finitely many primes $\ell$.

\item In the application, there won't be a single group $H$ possible, but a finite family $\mathcal{H}$ of cardinality depending on the factorisation of $2g$, namely $\mathcal{H} = \{ \op{Sym}_r : r \mid 2g \}$ and $\mathcal{H}= \{ \mathbb{Z}/r\mathbb{Z} : r \text{ is prime and }r \mid 2g \}$ for classes $\mathcal{C}_2$ and $\mathcal{C}_3$ respectively.
\end{enumerate}
\end{remark}

\newcommand{\SetOfCharacters}{\Gamma}
\newcommand{\SingleCharacter}{\gamma}

The example we will give will be semistable from which we will immediately conclude that $\varphi$ is trivial (except for small primes). We briefly outline an approach that weakens this assumption and only requires that $\varphi$ is unramified at all places above $\ell$. Suppose that $\phi : M \to H$ is fixed.
Since $\varphi$ factors through $\rho_\ell$, it has bounded ramification and degree and so, unless $\varphi$ is trivial,
\[
\SetOfCharacters:=\{ \SingleCharacter : \op{Gal}(\overline{K}/K) \to H \, \mid \SingleCharacter \text{ has conductor dividing } \mathfrak{N}_A\infty, \; \gamma \text{ is nontrivial} \}
\]
is a finite set containing $\varphi$ (although determination of $\SetOfCharacters$ may be nontrivial). 

\begin{lemma}\label{lem: general phi}

Suppose that, for each $\SingleCharacter \in \SetOfCharacters$, there exists $h_\SingleCharacter \in \operatorname{Im} \SingleCharacter$ such that $\phi(g)=h_\SingleCharacter$ implies $\operatorname{tr}(g)=0$. For each $\SingleCharacter \in \SetOfCharacters$ choose such an $h_\SingleCharacter$ and a place $v_{\SingleCharacter} \nmid \mathfrak{N}_A$ such that $\SingleCharacter(\op{Fr}_{v_{\SingleCharacter}}) =h_\SingleCharacter$. If $G_\ell \subseteq M$, then one of the following holds:
\begin{enumerate}
    \item $v_{\SingleCharacter} \mid \ell$ for some $\SingleCharacter \in \SetOfCharacters$; or
    \item $\varphi$ is trivial; or
    \item $\ell \mid \prod_{\SingleCharacter \in \SetOfCharacters} \op{tr}(\op{Fr}_{v_{\SingleCharacter}})$.
\end{enumerate}
\end{lemma}

\begin{proof}
Suppose $v_{\SingleCharacter} \nmid \ell$ for all $\SingleCharacter \in \SetOfCharacters$ and $\varphi$ is nontrivial. We know that $\SingleCharacter := \phi \circ \rho_\ell$ is in $\SetOfCharacters$. By definition of $h_\gamma$ and $v_{\SingleCharacter}$ we have $h_{\SingleCharacter} =\SingleCharacter(\op{Fr}_{v_\SingleCharacter}) = \phi(\rho_\ell(\operatorname{Fr}_{v_\SingleCharacter}))$, hence $\operatorname{tr} \rho_\ell(\operatorname{Fr}_{v_\SingleCharacter}) = 0$ and so $\ell \mid \op{tr}(\op{Fr}_{v_{\SingleCharacter}})$.
\end{proof}

We delay the discussion of finding suitable elements $h_\SingleCharacter$ to the relevant subsections. Instead we observe that for a given pair $(\SingleCharacter,h_\SingleCharacter)$ as in \Cref{lem: general phi}, there is a positive density of possible places $v_\SingleCharacter$ exist such that $\SingleCharacter(\op{Fr}_{v_{\SingleCharacter}}) =h_\SingleCharacter$ by Chebotarev's theorem. Moreover, we show momentarily in \Cref{lemma: trace 0 has density 0} that we can further impose that our choice of Frobenius has nonzero trace in order to obtain a nontrivial divisibility restriction on $\ell$.

\begin{lemma}\label{lemma: trace 0 has density 0}
The set $T$ of places $v$ of $K$ such that $\operatorname{tr} \operatorname{Fr}_v =0$ has density $0$.
\end{lemma}

\begin{proof}
   Let $\ell$ be any prime. The density of $T$ is at most the density of places $v$ for which $\operatorname{tr}(\operatorname{Fr}_v) \equiv 0 \pmod{\ell}$, which by Chebotarev's theorem is $\tfrac{\#\{h \in G_\ell \, : \, \operatorname{tr} h = 0\}}{\#G_\ell}$.
   For $\ell \gg_A 0$, by assumption we have $G_\ell=\operatorname{GSp}_{2g}(\mathbb{F}_\ell)$, and it is easy to see that for all primes $\ell \gg 0$ we have
   \[
   \frac{\#\{h \in \operatorname{GSp}_{2g}(\mathbb{F}_\ell) : \operatorname{tr} h = 0\}}{\#\operatorname{GSp}_{2g}(\mathbb{F}_\ell)} \leq \frac{2}{\ell} \xrightarrow{\ell \rightarrow \infty } 0.
   \]
\end{proof}

\subsection{Geometric reducibility and class $\mathcal{C}_1$}\label{subsubsect: geometric reducibility example}
In this section we take $K=\mathbb{Q}$. We assume that $G_\ell$ acts reducibly on $A[\ell] \otimes \overline{\mathbb{F}}_\ell$ and explain how to bound $\ell$. In particular, this covers the case when $G_\ell$ is contained in a maximal subgroup of $\operatorname{GSp}_{2g}(\mathbb{F}_\ell)$ of class $\mathcal{C}_1$.

\begin{lemma}\label{lemma: JHdet}
Let $A/\mathbb{Q}$ be an abelian variety, $\ell$ a prime number, and $W$ a Jordan--H\"{o}lder factor of $A[\ell] \otimes \overline{\mathbb{F}}_\ell$.

\begin{enumerate}
\item The character $\det W$ is unramified outside $S_\ell$
and is (at worst) tamely ramified at $\ell$.

\item If $A/\mathbb{Q}$ has semistable reduction at $\ell$, then $\det W|_{I_\ell}=\chi_\ell^r$ for some $0 \leq r \leq \dim W $, hence $\det W=\varepsilon \cdot \chi_\ell^r$ for some $0 \leq r \leq \dim W$ with $\varepsilon$ ramified at most at the primes of bad unstable reduction of $A$.
In particular, if $A/\mathbb{Q}$ is semistable everywhere, then $\det W=\chi_\ell^r$ for some $0 \leq r \leq \dim W$.
\end{enumerate}
\end{lemma}

\begin{proof}
(1) is a special case of \Cref{lem: semistable phi}(1) since the image of $\det W$ has order prime to $\ell$.

(2): From (1) and class field theory (in this case, the Kronecker--Weber theorem), $\det W|_{I_\ell}$ factors through $\op{Gal}(\mathbb{Q}(\zeta_\ell)/\mathbb{Q})$, whose character group is generated by the cyclotomic character. 
Let $U$ be the simple Jordan--H\"{o}lder factor of the inertia module $A[\ell]$ whose semisimplification over $\overline{\mathbb{F}}_\ell$ contains $W|_{I_\ell}$ and write $n=\dim U$. By \Cref{thm: JH fundamental chars}, the eigenvalues of $U$ (and hence of $W|_{I_\ell}$) are of the form $\psi_n^j$ for some $0 \leq j \leq \frac{\ell^n-1}{\ell-1}$. Therefore $\det W|_{I_\ell} = \psi_n^{n_W}$ where $0 \leq n_W \leq \dim(W) \frac{\ell^n-1}{\ell-1}$. Since $\det W|_{I_\ell}$ is a power of $\chi_\ell=\psi_n^{(\ell^n-1)/(\ell-1)}$ as shown above, it follows $\det W|_{I_\ell} = \chi_\ell^r$ with $0 \leq r \leq \dim(W)$. The rest of (2) follows easily.
\end{proof}

Note that \Cref{lemma: JHdet}(2) applies to any given abelian variety $A/\mathbb{Q}$ for all but finitely many primes $\ell$. With the notation of the lemma, let $k \geq 1$ be such that $\det W$ has image contained in $\mathbb{F}_{\ell^k}^\times$. By identifying $\mathbb{F}_{\ell^k}$ with the residue field of $\mathbb{Z}[\zeta_{\ell^k-1}]$ at a prime of characteristic $\ell$, we consider $\varepsilon$ as the reduction of a character $\varepsilon : \operatorname{Gal}(\overline{\mathbb{Q}}/\mathbb{Q}) \to \mu_{\ell^k-1} \subset \mathbb{Z}[\zeta_{\ell^k-1}]$.
When $p$ is a prime where $\varepsilon$ is unramified, we write $\varepsilon(p)$ to denote $\varepsilon(\operatorname{Fr}_p)$.

\begin{lemma}\label{lem: JHelim}
Suppose that $A$ is semistable at $\ell$. Let $W$ be a Jordan--H\"older factor of $A[\ell] \otimes \overline{\mathbb{F}}_\ell$ and write $\det W = \varepsilon \cdot \chi_\ell^r$ as in \Cref{lemma: JHdet}(2). For a prime $p$ of good reduction of $A$, write $f_p(x)=\prod_{i=1}^{2g}(x-\mu_i)$ in $\overline{\mathbb{Q}}[x]$ and set 
\[
g_p(x) := \prod_{\substack{I \subseteq \{1,\ldots,2g\} \\ \#I=\dim W}}\left(x-\prod_{i \in I} \mu_i\right) \in \mathbb{Z}[x].
\]
Then $\ell \mid p \cdot N_{\mathbb{Q}(\zeta_{\ell^k-1})/\mathbb{Q}}\left(g_p(\varepsilon(p) \cdot p^r)\right)$. Moreover, if $r \neq {\dim W/2}$, then $g_p(\varepsilon(p) \cdot p^r) \neq 0$.
\end{lemma}

\begin{proof}
We can assume $\ell \neq p$. Since $W$ is a subrepresentation of the semisimplification of $A[\ell]$, its eigenvalues for $\op{Fr}_p$ are roots of $f_p(x)$, and the determinant of $\op{Fr}_p$ acting on $W$ is the product of $\dim W$ roots of $f_p(x) \in \overline{\mathbb{F}}_\ell[x]$. In particular, we have $g_p( (\det W)(\operatorname{Fr}_p))=0$ in $\overline{\mathbb{F}}_\ell$. The lemma then follows from the fact that $(\det W)(\op{Fr}_p)$ is the reduction modulo (a prime above) $\ell$ of $\varepsilon(p) \cdot \chi_\ell(\op{Fr}_p)^r = \varepsilon(p) p^r$. Since the roots of $f_p(x)$ are $p$-Weil numbers of weight $1$ and $|\varepsilon(p)|=1$, the quantity $g_p(\varepsilon(p) \cdot p^r)$ can be zero only when $r=\frac{1}{2}\dim W$.
\end{proof}

\begin{remark}\label{rmk: det}
By running over all possible values of $r$ and the finitely many characters $\varepsilon$ with conductor dividing that of $A$ we can produce a finite set of primes $\ell$ that can have a Jordan--H\"{o}lder factor of odd dimension. Furthermore, since $A[\ell] \cong A^\vee[\ell](1)$ we only have to check $0 \leq r \leq \frac{1}{2}\dim W$.
\end{remark}

\begin{remark}
The key ingredient we used in \Cref{lem: JHelim} is that $\det W$ is the product of a character with bounded ramification and a bounded power of the $\bmod{\, \ell}$ cyclotomic character. If one wished to generalise to number fields $K$, then one could impose conditions on the ray class field extension $K(\mathfrak{l}\infty)/K(\infty)$ for all $\mathfrak{l} \mid \ell$ to ensure $\det W\mid_{I_\mathfrak{l}}$ is a bounded power of the level $1$ fundamental character (as in \Cref{lemma: JHdet}(2)) and then control the ramification of $\ell$ to relate this to the cyclotomic character.
\end{remark}

When $\dim A \leq 3$, \Cref{lem: JHelim} and \Cref{rmk: det} allow us to bound the values of $\ell$ for which $A[\ell] \otimes \overline{\mathbb{F}}_\ell$ has a nontrivial Jordan--H\"older factor $W$ unless $\dim W = 2$ and $\det W = \varepsilon \chi_\ell$. We will deal with this case using Serre's modularity conjecture (now a theorem of Khare--Wintenberger). For simplicity we only consider the case $\varepsilon=1$ that we will need for the example in \Cref{sect: Example}; adapting this to general $\varepsilon$ is straightforward (simply apply the considerations below to $W \otimes \varepsilon^{-1}$).

\newcommand{\ModForm}{\tilde{f}}

\begin{theorem}[{\cite[Theorem 1.2]{MR2551763} and \cite[Lemma 3.9]{MR4732686}}]\label{thm: SerreModularity}
Suppose that $A$ has semistable reduction at $\ell$. Let $W$ be a Jordan--H\"older factor of $A[\ell] \otimes \overline{\mathbb{F}}_\ell$ with $\dim W=2$ and $\det W=\chi_\ell$. There exists a weight $2$ newform $\ModForm =q+\sum_{n\geq 2} a_n(\ModForm)q^n \in S_2(\Gamma_0(N))^{\operatorname{new}}$ of level $N \mid \mathfrak{N}_A$ and a prime $\mathfrak{l} \mid \ell$ of the Hecke eigenfield $K_{\ModForm} := \mathbb{Q}(a_n(\ModForm) : n \in \mathbb{N})$ such that
\[
\det(t\cdot \op{Id} - \op{Fr}_p \mid W) \equiv t^2-a_p(\ModForm)t + p \pmod{\mathfrak{l}}
\]
for all primes $p \nmid N_A\ell$.
\end{theorem}

\begin{remark}
If $N=\mathfrak{N}_A$, then the semisimplification of $A[\ell]$ (which is the direct sum of its Jordan--H\"older factors) necessarily splits as $W \oplus W'$ with $W'$ unramified. Hence $W'$ is trivial, a case which we have already dealt with. We may therefore suppose $N \neq \mathfrak{N}_A$.
\end{remark}

For each proper divisor $N$ of $\mathfrak{N}_A$, let $B_N$ be an eigenbasis of newforms in $S_2(\Gamma_0(N))^{\op{new}}$. Define
\[
H_{N, p}(t) := 
\prod_{\ModForm \in B_N} \prod_{\sigma : K_{\ModForm} \hookrightarrow \mathbb{C}} \left( t^2-\sigma(a_p(\ModForm))t+p \right).
\]
This is the characteristic polynomial of the Hecke operator $T_p$ acting on $S_2(\Gamma_0(N))^{\operatorname{new}}$ and can be easily computed for moderate values of $N$. 

\begin{lemma}\label{lemma: eliminating two dimensional JH factors}
Suppose that $\ell$ is a prime of semistable reduction and $A[\ell]\otimes \overline{\mathbb{F}}_\ell$ has a Jordan--H\"older factor $W$ with $\dim W =2, \det W = \chi_\ell$. Then
        \[
        \ell \mid p \prod_{\substack{N \mid \mathfrak{N}_A \\ N \neq \mathfrak{N}_A}}\operatorname{Res}\left( H_{N,p}(t), f_p(t) \right)
        \]
        for all primes $p$ of good reduction.
\end{lemma}
\begin{proof}

Let $\ModForm$ be the newform of \Cref{thm: SerreModularity} corresponding to $W$. The definitions imply that for every prime $p \neq \ell$ of good reduction of $A$, the characteristic polynomial $f_p(t) \in \overline{\mathbb{F}}_\ell[t]$ has the factor $t^2-a_p(\ModForm)t + p$ in common with $H_{N, p}(t)$. 
\end{proof}

\begin{remark}\label{rmk: proving irreducibility}
The previous lemma only applies to semistable primes (this is used to deduce that the newform of \Cref{thm: SerreModularity} has weight 2, see the proof of \cite[Lemma 3.9]{MR4732686}). For each of the finitely many non-semistable primes $\ell$, the irreducibility of $A[\ell]$ can be proved by finding a prime $p$ such that $f_p(x)$ is irreducible in $\mathbb{F}_\ell[x]$; such a prime $p$ always exists if $G_\ell=\operatorname{GSp}_{2g}(\mathbb{F}_\ell)$.
\end{remark}

\begin{remark}
The previous remark explains how to show irreducibility of $A[\ell]$ over $\mathbb{F}_\ell$ for a given prime $\ell$. We now discuss the more general problem of showing geometric irreducibility.
By Schur's lemma, the representation $A[\ell] \otimes_{\mathbb{F}_\ell} \overline{\mathbb{F}}_\ell$ is irreducible if and only if
\[
\operatorname{End}_{\overline{\mathbb{F}}_\ell[G_\ell]}(A[\ell] \otimes \overline{\mathbb{F}}_\ell) \cong \overline{\mathbb{F}}_\ell;
\]
since 
\[
\operatorname{End}_{\overline{\mathbb{F}}_\ell[G_\ell]}(A[\ell] \otimes \overline{\mathbb{F}}_\ell) \cong \operatorname{End}_{\mathbb{F}_\ell[G_\ell]}(A[\ell]) \otimes_{\mathbb{F}_\ell} \overline{\mathbb{F}}_\ell,
\]
we see that geometric irreducibility is equivalent to the fact that the commutant of $G_\ell$ in $\operatorname{End}_{\mathbb{F}_\ell}(A[\ell])$ is $\mathbb{F}_\ell$. Suppose that $A[\ell]$ is irreducible. By Schur's lemma again, the endomorphism ring
$
\operatorname{End}_{\mathbb{F}_\ell[G_\ell]}(A[\ell])
$
is a division algebra.
Since it is also finite, it is a (finite) field by Wedderburn's little theorem. If $A[\ell] \otimes_{\mathbb{F}_\ell} \overline{\mathbb{F}}_\ell$ is reducible, it follows from the above that $\operatorname{End}_{\mathbb{F}_\ell[G_\ell]}(A[\ell]) \cong \mathbb{F}_{\ell^r}$ for some $r \geq 2$. Thus, if $A[\ell]$ is irreducible but geometrically reducible, $G_\ell$ acts $\mathbb{F}_{\ell^r}$-linearly on $A[\ell]$ for some $r \geq 2$, and in particular, the characteristic polynomial of every element in $G_\ell$ is the norm from $\mathbb{F}_{\ell^r}$ to $\mathbb{F}_{\ell}$ of some polynomial with coefficients in $\mathbb{F}_{\ell^r}$. To show geometric irreducibility it is then sufficient to exhibit a characteristic polynomial of Frobenius that is not of this form. For example, any polynomial with precisely one root in $\mathbb{F}_\ell$ has this property: if a polynomial $p(x) \in \mathbb{F}_\ell[x]$ is the norm of a polynomial in $\mathbb{F}_{\ell^r}[x]$ and $\mu \in \overline{\mathbb{F}}_\ell$ is a root of $p(x)$, then the multiset $\{\mu, \mu^\ell, \ldots, \mu^{\ell^{r-1}}\}$ is contained in the multiset of roots of $p(x)$, so $p(x)$ cannot have precisely one root in $\mathbb{F}_\ell$.
Primes $\ell$ for which geometric irreducibility fails may also be bounded using \Cref{prop_EasyCases}. Unfortunately, the bound one gets in this way is too large to be useful in practice.
\end{remark}

\subsection{Class $\mathcal{C}_2$: imprimitive groups}

If $A[\ell]$ is an imprimitive Galois representation, with image contained in the $\mathcal{C}_2$-group $M$, then there is a block decomposition $A[\ell]=\bigoplus_{i=1}^r W_i$ with $r \geq 2$ preserved by the action of $G_\ell$. We can hence define a homomorphism $\varphi: \operatorname{Gal}(\overline{K}/K) \xrightarrow{\rho_\ell} G_\ell \subseteq M \xrightarrow{\phi} \op{Sym}(\{W_1,\ldots, W_r\})$. The first part of the following lemma enables us to deal with individual primes (namely small primes and those dividing the conductor), whilst the third part concerns the rest when $A$ is semistable.

\begin{lemma}\label{lem: primEX}
With notation as above, the following hold.
\begin{enumerate}
\item  Suppose there is a prime $v \nmid \mathfrak{N}_A\ell\infty$ such that $f_v \bmod{\ell}$ is irreducible and $\op{tr} \, \rho_\ell(\op{Fr}_v) \not\equiv 0 \bmod{\ell}$. Then $A[\ell]$ is primitive.
\item Let $\mathfrak{l} \mid \ell$ and suppose that $A$ is semistable at $\mathfrak{l}$. Recall the function $\xi(r)$ from \Cref{eq: definition of the function xi}. If $\ell > e(\mathfrak{l} \mid \ell)\xi(r) +1,$ then $\varphi$ is unramified at $\mathfrak{l}$.
\item Suppose $A/K$ is semistable, $\mathfrak{N}_A$ is coprime to $\ell$, and $\ell> \max_{\mathfrak{l} \mid \ell} e(\mathfrak{l} \mid \ell)\xi(r)+1$. Then $\varphi$ is unramified at all finite places. In particular, if $K$ has narrow class number $1$, then $A[\ell]$ is primitive.
\end{enumerate}
\end{lemma}

\begin{proof}
(1) is proved as part of \cite[Lemma 7.2]{ZywinaExample} and (3) follows from (2) and \Cref{lem: semistable phi}(2) (noting that $\xi(r)\geq r$ hence $\ell \nmid \#\op{Sym}_r$) so it remains to prove (2); this is analogous to \cite[Lemma 7.3]{ZywinaExample} accounting for ramification degree.
\end{proof}

\begin{remark}
If $A/K$ is not semistable, one can use \Cref{lem: general phi} to bound the set of imprimitive $\ell$. We briefly explain how to construct the necessary elements $h_\SingleCharacter$. Note that we may assume that $H=\op{Im} \SingleCharacter \subset \op{Sym}_r$ is a transitive subgroup (else $G_\ell$ is reducible). Burnside's lemma then implies that $H$ contains an element $h_\SingleCharacter$ without fixed points which we claim works; indeed writing the matrix of $g \in M$ in block form, it is clear that if $\phi(g) = h_\SingleCharacter$ then $\op{tr} g=0.$

\end{remark}

\subsection{Class $\mathcal{C}_3$: field extension subgroups}

Consider $A[\ell] \cong \mathbb{F}_\ell^{2g}$ as endowed with the structure of an $\mathbb{F}_{\ell^r}$-vector space for some prime $r \mid 2g$ and let $M$ be the subgroup of $\mathbb{F}_{\ell^r}$-semilinear automorphisms of $A[\ell]$ (see \Cref{def: semilinear}). In this section we assume $G_\ell \subseteq M$.
Let $M_0$ be the normal, index-$r$ subgroup of $M$ that acts $\mathbb{F}_{\ell^r}$-linearly on $A[\ell]$.
First we wish to show that the eigenvalues of every element of $M \setminus M_0$ are all of a certain shape. This is no doubt well known, but we could not find a reference so we provide a proof.

\begin{lemma}\label{lemma: semilinear action}
    Let $\ell$ be a prime, $r, d$ be positive integers, and $V=\mathbb{F}_{\ell^r}^d$. Let $G = \operatorname{Gal}(\mathbb{F}_{\ell^r}/\mathbb{F}_\ell)$, with generator $\operatorname{Frob} : x \mapsto x^\ell$, and $M$ be the group of $\mathbb{F}_{\ell^r}$-semilinear automorphisms of $V$.
    \begin{enumerate}
        \item 
    The map
\[
\begin{array}{cccc}
\Phi : & V \otimes_{\mathbb{F}_\ell} \mathbb{F}_{\ell^r} &\to & V^G := \prod_{\sigma \in G} V \\
& v \otimes \lambda & \mapsto & (v\sigma(\lambda))_{\sigma \in G}
\end{array}
\]
is a $\operatorname{GL}_d(\mathbb{F}_{\ell^r})$-equivariant, $\mathbb{F}_\ell$-linear isomorphism, where $\operatorname{GL}_d(\mathbb{F}_{\ell^r})$ acts on $V^G$ via its diagonal action.
\item The natural representation of $\operatorname{GL}_d(\mathbb{F}_{\ell^r})$ on $V \otimes_{\mathbb{F}_\ell} \mathbb{F}_{\ell^r}$ is reducible.

\item Let $M_0=\operatorname{GL}_d(\mathbb{F}_{\ell^r})$, seen as an index-$r$ subgroup of $M$, and fix $g \in M \setminus M_0$. The eigenvalues of $g$, seen as an automorphism of the $dr$-dimensional $\mathbb{F}_\ell$-vector space $V$, are of the form $\lambda_i \zeta_r^{j}$ for certain $\lambda_1, \ldots, \lambda_{d} \in \overline{\mathbb{F}}_\ell^\times$, $1 \leq i \leq d$ and $0 \leq j \leq r-1$. 
In particular, $\operatorname{tr}(g)=0$.
    \end{enumerate}
\end{lemma}

\begin{proof}
For (1), that $\Phi$ is an isomorphism follows from the case $d=1$, which is a well-known fact in Galois theory. The equivariance with respect to $\operatorname{GL}_d(\mathbb{F}_{\ell^r})$ is a simple calculation.
(2) is a consequence of (1), because $\Phi$ is an isomorphism of $\operatorname{GL}_d(\mathbb{F}_{\ell^r})$-representations, and it is clear that $V^G$ (with the diagonal action) is reducible. For (3), the normal basis theorem gives the existence of $x \in \mathbb{F}_{\ell^r}$ such that $x, \operatorname{Frob}(x), \ldots, \operatorname{Frob}^{r-1}(x)$ is an $\mathbb{F}_\ell$-basis of $\mathbb{F}_{\ell^r}$. Write $e_i$, $1 \leq i \leq d$, for the canonical basis of $\mathbb{F}_{\ell^r}^d$ as an $\mathbb{F}_{\ell^r}$-vector space. 
Let $g \in M$. The multiset of eigenvalues of $g \otimes \operatorname{id}$ on $V \otimes_{\mathbb{F}_\ell} \mathbb{F}_{\ell^r}$ is given by $r$ copies of the multiset of eigenvalues of $g$ on $V$, so it suffices to compute the former. Note that $M$ is generated by $M_0$ together with the map
\[
     \begin{array}{cccc}
     \tau : & \mathbb{F}_{\ell^r}^d & \to & \mathbb{F}_{\ell^r}^d \\
     & {}^t(x_1, \ldots, x_d) & \mapsto & {}^t(x_1^\ell,\ldots,x_d^{\ell}).
     \end{array}
\]
We start by computing the action of $\tau$. For the purposes of this calculation, we identify $G$ with $\mathbb{Z}/r\mathbb{Z}$ by sending $\operatorname{Frob}^k$ to $k$. Since the trace pairing $(u,v) \mapsto \operatorname{tr}_{\mathbb{F}_{\ell^r}/\mathbb{F}_\ell}(uv)$ is nondegenerate, there exists $y \in \mathbb{F}_{\ell^r}$ such that 
\[
\operatorname{tr}_{\mathbb{F}_{\ell^r}/\mathbb{F}_\ell}(y \operatorname{Frob}^j(x))=\begin{cases}
  1 \text{ if } j \equiv 0 \pmod{r}  \\
  0 \text{ otherwise}.
\end{cases}
\]
Let $v_{i, j, k} := \sum_{m=1}^r \operatorname{Frob}^{j}(x)  \operatorname{Frob}^m(y)e_i \otimes \operatorname{Frob}^{m-k}(x)$, for $1 \leq i \leq d$ and $j,k \in \mathbb{Z}/r\mathbb{Z}$. Then
\[
\begin{aligned}
\Phi(v_{i,j,k}) & = \left( \sum_{m=1}^r \operatorname{Frob}^{j}(x) \operatorname{Frob}^m(y) \operatorname{Frob}^{m-k+n}(x) \, e_i \right)_{n \in \mathbb{Z}/r\mathbb{Z}} \\ & =  \left( \sum_{m=1}^r \operatorname{Frob}^{j}(x) \operatorname{tr}_{\mathbb{F}_{\ell^r/\mathbb{F}_\ell}}(y \operatorname{Frob}^{n-k}(x)) \,  e_i \right)_{n \in \mathbb{Z}/r\mathbb{Z}}
\end{aligned}
\]
is the vector having component $\operatorname{Frob}^{j}(x)e_i$ along the $k$-th copy of $V$ and zero elsewhere. It is then clear that the $v_{i, j, k}$ (for $1 \leq i \leq d$ and $j,k \in \mathbb{Z}/r\mathbb{Z}$) form an $\mathbb{F}_\ell$-basis of $V \otimes_{\mathbb{F}_\ell} \mathbb{F}_{\ell^r}$. On the other hand, $\tau \otimes \operatorname{id}$ sends $v_{i, j, k}$ to $v_{i, j+1, k+1}$. Thus, for a suitable ordering of the basis vectors $\Phi(v_{i,j,k})$, the $\mathbb{F}_\ell$-linear action of $\tau \otimes \operatorname{id}$ on $V^G \cong V \otimes_{\mathbb{F}_\ell} \mathbb{F}_{\ell^r}$ is represented by a block matrix of the form
\[
\begin{pmatrix}
0 & 0 & \cdots & 0 & R \\ 
R & 0 & \cdots & 0 & 0 \\ 
0 & R & \cdots & 0 & 0 \\ 
0 & 0 & \ddots & 0 & 0 \\ 
0 & 0 & \cdots & R & 0 \\ 
\end{pmatrix}
\]
where each block $R$ is itself a $dr \times dr$ permutation matrix. On the other hand, for any $g_0 \in M_0$ the action of $g_0 \otimes \operatorname{id}$ on the basis $\Phi(v_{i,j,k})$ is given by a block-diagonal matrix $\operatorname{diag}(B, B, \ldots, B)$ for some $dr \times dr$ matrix $B$. Since every $g \in M \setminus M_0$ is of the form $g = g_0 \tau^i$ for some $i \not \equiv 0 \pmod{r}$, we conclude that the matrix representation of $g$ is a cyclic block-permutation matrix whose nonzero $dr \times dr$ blocks are all equal (and all off the diagonal). It is an easy exercise in linear algebra to prove that the characteristic polynomial of such a matrix is a polynomial in $t^r$, from which the claim follows.
\end{proof}

With this lemma in hand, we can now turn our attention to how far $G_\ell$ is from being contained in $M_0$. Note that if $G_\ell \subseteq M_0$, then \Cref{lemma: semilinear action}(2) implies that $A[\ell] \otimes \overline{\mathbb{F}}_\ell$ is reducible which is a case we have already handled.

\begin{lemma}\label{lemma: ramification of C3 character}
Let $M, M_0, r$ be as above and assume $G_\ell \subseteq M$. Consider the character
\[
\varphi : \operatorname{Gal}(\overline{K}/K) \to G_\ell \subseteq M \to M/M_0 \cong \mathbb{Z}/r\mathbb{Z}.
\]
\begin{enumerate}
\item Suppose $A$ is semistable at $\mathfrak{l}\mid \ell$. If $\ell>e(\mathfrak{l}\mid\ell)r+1$, then $\varphi$ is unramified at $\mathfrak{l}$.
\item If $A/K$ is semistable and $\ell >\max_{\mathfrak{l} \mid \ell} e(\mathfrak{l} \mid \ell)r +1$ then $\varphi$ is unramified at all finite places. If $K$ has narrow class number $1$, then $A[\ell]$ is geometrically reducible. 
\end{enumerate}
\end{lemma}

\begin{proof}
(2): \Cref{lem: semistable phi}(2) and (1) imply that $\varphi$ is everywhere unramified. Thus, if $K$ has narrow class number $1$, the map $\varphi$ is trivial and hence $A[\ell]$ is geometrically reducible by \Cref{lemma: semilinear action}(2). We now prove (1). Since $A$ is semistable, every eigenvalue of $A[\ell]$ corresponds to $\psi_{n_i}^{k_i}$ for certain positive integers $n_i, k_i$ such that $0 \leq k_i \leq e(\mathfrak{l}\mid \ell)\frac{\ell^{n_i}-1}{\ell-1}$; in particular $\varphi\mid_{I_\mathfrak{l}}$ factors through the tame inertia group.
Fix a choice of lift of a generator of tame inertia $x$. If $\varphi(x) \neq 1$, then by \Cref{lemma: semilinear action}(3) there exist two distinct eigenvalues  $\mu_1=\psi_{n_1}^{k_1}(x)$ and $\mu_2=\psi_{n_2}^{k_2}(x)$ of $\rho_\ell(x)$ such that $\mu_1^r=\mu_2^r$. Since $x$ is a generator, this is true for the characters hence $\psi_{n_1}^{rk_1}=\psi_{n_2}^{rk_2}$. Letting $n=\op{lcm}(n_1,n_2)$ we can rewrite this as $\psi_n^{rk_1\frac{\ell^n-1}{\ell^{n_1}-1}} = \psi_n^{rk_2\frac{\ell^n-1}{\ell^{n_2}-1}}$. Since $\psi_n$ has order $\ell^n-1$ and $\ell> re(\mathfrak{l}\mid \ell)+1$, this forces $rk_1\frac{\ell^n-1}{\ell^{n_1}-1}=rk_2\frac{\ell^n-1}{\ell^{n_2}-1}$ which is a contradiction since $\mu_1, \mu_2$ are distinct.
\end{proof}

\begin{remark}\label{rmk: conclusion in C3 case}
Similarly to the primitive case, we can use \Cref{lem: general phi} when $A/K$ is not semistable. In this setting, letting $\phi : M \to M/M_0$ be the quotient map, one can take $h_\SingleCharacter$ to be any nontrivial element of $\mathbb{Z}/r\mathbb{Z}$; the fact that $\phi(g)=h_\SingleCharacter$ implies $\operatorname{tr} g=0$ follows from \Cref{lemma: semilinear action}(3).
\end{remark}

\subsection{Class $\mathcal{C}_4$: tensor product groups}

Suppose that $G_\ell$ is contained in a maximal subgroup of class $\mathcal{C}_4$.
Let $v$ be a place of good reduction of $A$ such that $f_v(x)$ has no repeated roots. The proof of \Cref{prop:tensor products} shows that $\op{Fr}_v$ has distinct eigenvalues $\lambda_1, \lambda_2, \lambda_3, \lambda_4$ satisfying $N_{\mathbb{Q}(v)/\mathbb{Q}}(\lambda_1^2-\lambda_2\lambda_3) \equiv 0 \pmod{\ell}$ or $N_{\mathbb{Q}(v)/\mathbb{Q}}(\lambda_1\lambda_2-\lambda_3\lambda_4) \equiv 0 \pmod{\ell}$ (in the latter case, $\lambda_1 \lambda_2 \neq \lambda_3 \lambda_4$ as elements of $\overline{\mathbb{Q}}^\times$). Thus, for each $v$ of good reduction we obtain
\begin{equation}\label{eq: divisibility relation for tensor product case}
    \ell \mid p_v \cdot \prod_{\substack{\lambda_1, \lambda_2,\lambda_3 \\ \text{distinct}}} N_{\mathbb{Q}(v)/\mathbb{Q}}(\lambda_1^2-\lambda_2\lambda_3)  \cdot \prod_{\substack{\lambda_1, \lambda_2,\lambda_3, \lambda_4 \\ \text{distinct} \\ \lambda_1\lambda_2 \neq \lambda_3\lambda_4}} N_{\mathbb{Q}(v)/\mathbb{Q}}(\lambda_1\lambda_2-\lambda_3\lambda_4),
\end{equation}
where each product ranges over the roots $\lambda_i$ of $f_v(x)$ with the given properties. If we assume that the roots of $f_v(x)$ are multiplicatively independent (which is true for a density-1 set of places $v$), the right-hand side of \Cref{eq: divisibility relation for tensor product case} is nonzero, hence we get a divisibility bound on $\ell$.

\subsection{Class $\mathcal{S}$, constant groups}\label{subsect: example constant groups}
Suppose that $G_\ell$ is contained in a maximal subgroup of class $\mathcal{S}$ `of constant type'. We wish to bound $\ell$.
For simplicity of exposition we focus on the case $\dim A = 3$, but the arguments below extend in principle to any dimension, since all we need is an upper bound on the size (or even just exponent) of such groups, which we can get for example from \Cref{prop_UnconditionalConstantGroups}.
The proof of \cite[Theorem 6.6(1)]{MR4695874} gives the following.

\begin{theorem}\label{thm: lower bound exponent projective image}
Suppose $A/K$ is semistable at $\mathfrak{l} \mid \ell$. If $\ell > 2e(\mathfrak{l} \mid \ell)$, then the exponent of the group $\mathbb{P}G_\ell$ is at least $\frac{\ell-1}{e(\mathfrak{l} \mid \ell)}$.
\end{theorem}

Note that any abelian variety $A/K$ has semistable reduction at all but finitely many primes, and that the inequality $\ell > 2e(\mathfrak{l} \mid \ell)$ is satisfied for example by all primes $\ell>2$ that are unramified in $K$ (so in particular by all but finitely many primes).

\begin{remark}
Note that we can drop the semistability assumption and get a lower bound on the exponent of $\mathbb{P}G_\ell$ that only depends on $\ell$ and not the choice of prime $\mathfrak{l}\mid \ell$. Indeed it is well known \cite[Exposé IX, Proposition 4.7]{MR354656} that $A$ becomes everywhere semistable over an extension $L/K$ of degree bounded by $\#\op{GSp}_{2g}(\mathbb{Z}/12\mathbb{Z)}$. Let $\mathfrak{l}$ be a prime of $L$ above $\ell$. Since $A$ is semistable and  $e(\mathfrak{l} \mid \ell) \leq [L:\mathbb{Q}] \leq \#\op{GSp}_{2g}(\mathbb{Z}/12\mathbb{Z}) \cdot [K:\mathbb{Q}]$, applying \Cref{thm: lower bound exponent projective image} to $A_L$ we obtain the uniform lower bound $\left\lfloor \frac{\ell-1}{\#\op{GSp}_{2g}(\mathbb{Z}/12\mathbb{Z}) \cdot [K:\mathbb{Q}]} \right\rfloor$ for the exponent of $\mathbb{P}G_\ell$.
\end{remark}

Fix a prime $\ell$ to which \Cref{thm: lower bound exponent projective image} applies.
Let $E$ be a class $\mathcal{S}$, constant subgroup of $\operatorname{GSp}_6(\mathbb{F}_\ell)$. The order of $\mathbb{P}E$ is at most $1 \, 209 \,600$, corresponding to the entry $2^{\boldsymbol{\cdot}}J_2$ in \cite[Table 8.29]{MR3098485}. Note that the maximal subgroup $2^{\boldsymbol{\cdot}}J_2$ of $\operatorname{Sp}_6(\mathbb{F}_\ell)$ extends to a subgroup $E$ of $\operatorname{GSp}_6(\mathbb{F}_\ell)$ of order $(\ell-1) \cdot 2 \cdot \#J_2$; the projective image of $E$ has order $2 \cdot \#J_2$. The inequality $\#\mathbb{P}G_\ell \geq \frac{\ell-1}{e(\mathfrak{l} \mid \ell)}$  gives $\ell \leq 2e(\mathfrak{l} \mid \ell) \cdot \#J_2+1$. 
While this bound may seem large, we have a quick test to eliminate possible subgroups $E$ for a \textit{fixed} prime $\ell$.

\begin{lemma}
Let $v \nmid \ell$ be a place of good reduction for $A/K$ and let $\overline{\mu}_1, \overline{\mu}_2 \in \overline{\mathbb{F}}_\ell^{\times}$ be distinct roots of the reduction of the characteristic polynomial of $\op{Fr}_v$. Let $N$ be the multiplicative order of $\overline{\mu}_1/\overline{\mu}_2$. If $G_\ell \subset E$, then $N$ divides the exponent of $\mathbb{P}E$.
\end{lemma}

\begin{proof}
If $G_\ell \subseteq E$, then $\mathbb{P}G_\ell \subseteq \mathbb{P}E$ so $g^m$ is scalar for all $g \in G_\ell$, where $m$ is the exponent of $\mathbb{P}E$. In particular, if $\lambda_1, \lambda_2$ are any two eigenvalues of $g$ then $(\lambda_1/\lambda_2)^m=1$. 
\end{proof}

For $g=3$, one sees from \cite[Table 8.29]{MR3098485} that the prime factors of $\#\mathbb{P}E$ are contained in the set $P := \{2,3,5,7,13\}$. Thus, it suffices to find a suitable place $v$ such that the corresponding $N$ has a prime factor outside $P$ to show that $G_\ell$ is not contained in \textit{any} maximal constant group in class $\mathcal{S}$.

\begin{remark}\label{rmk: on ruling out the constant groups}
\begin{enumerate}
\item One can rephrase the above by saying that if $G_\ell \subseteq E$, then the polynomials $x^{\op{exp} \mathbb{P}E}-1$ and 
        \[
        h_p(x) = \prod_{\substack{\lambda_1, \lambda_2 \in \overline{\mathbb{Q}}^\times \\ f_v(\lambda_1)=f_v(\lambda_2)=0 \\ \lambda_1 \neq \lambda_2}} (x - \lambda_1/\lambda_2)
        \]
        have a common root modulo $\ell$, and therefore the resultant of $x^{\op{exp} \mathbb{P}E}-1$ and $h_p(x)$ (a rational number) has positive $\ell$-adic valuation. This gives a divisibility bound on $\ell$. However, since $\op{exp} \mathbb{P}E$ is very large, this method is hard to implement in practice.
        \item The method described above cannot be used to rule out that $G_3$ is contained in a maximal subgroup of class $\mathcal{S}$, because the only prime factors of $\#\operatorname{GSp}_6(\mathbb{F}_3)$ are $2, 3, 5, 7, 13$. To handle this case, one can explicitly compute the only maximal class $\mathcal{S}$ subgroup $E$ of $\operatorname{GSp}_6(\mathbb{F}_3)$ and list the characteristic polynomials of its elements (there are 8 of them). Finding a single place $v$ such that $f_v(x) \bmod 3$ is not in this list shows that $G_3 \not \subseteq E$.
        \item Suppose that $A$ admits a place $v$ of bad semistable reduction with Tamagawa number $c_v$. Let $\ell$ be a prime not dividing $p_vc_v$. Well known arguments of Grothendieck (spelled out in \cite[end of the proof of Theorem 10.7]{MR3576113}) imply that an inertia group at $v$ acts on $A[\ell]$ via a unipotent operator that is nontrivial modulo $\ell$. In particular, this shows that $\mathbb{P}G_\ell$ contains nontrivial elements of $\ell$-power order, which rules out all the constant groups for $\ell > 13$ and $\ell \nmid c_vp_v$. The few remaining primes can then be handled as explained above. Note however that $A$ may have (potentially) good reduction everywhere, so this method -- while often applicable -- does not always give a divisibility bound on $\ell$.
    \end{enumerate}
\end{remark}

\subsection{Class $\mathcal{S}$, groups of Lie type}
For $g=3$ there is only one such maximal subgroup $M \subset \operatorname{GSp}_6(\mathbb{F}_\ell)$ up to conjugacy, namely, the subgroup generated by the scalars and by the image of $\operatorname{GL}_2(\mathbb{F}_\ell)$ acting via the fifth power of its standard representation. In this case, the eigenvalues of every element in $M$ are of the form $\{\lambda^i\mu^j : i+j=5\}$ for certain $\lambda, \mu \in \overline{\mathbb{F}}_\ell^\times$. In particular, if $v$ is a place of good reduction, there exist distinct eigenvalues $\lambda_1=\lambda^4\mu, \lambda_2=\lambda^5, \lambda_3=\lambda^3\mu^2$ of $\op{Fr}_v$ with $\lambda_2^2=\lambda_1\lambda_3$. We obtain 
\begin{equation}\label{eq: divisibility for groups of Lie type}
\ell \mid p_v  \cdot \prod_{\substack{\lambda_1, \lambda_2,\lambda_3 \\ \text{distinct}}} N_{\mathbb{Q}(v)/\mathbb{Q}}(\lambda_1^2-\lambda_2\lambda_3),    
\end{equation}
a sharper divisibility bound than the one we have for groups of class $\mathcal{C}_4$.
\begin{remark}
    If $\ell$ is a place of semistable reduction, one could also argue as in \cite[\S3.6]{MR1969642}, which considers (in the case $g=2$) the third power of the standard representation of $\operatorname{GL}_2(\mathbb{F}_\ell)$.
\end{remark}

For general $g$, we merely note that the arguments of \Cref{sect_LieGroups} give an effective divisibility bound in all cases.

\subsection{Classes $\mathcal{C}_6$ and $\mathcal{C}_7$}
With reference to \Cref{thm_Aschbacher}, for $g=3$ there are no such maximal subgroups of $\operatorname{GSp}_{2g}(\mathbb{F}_\ell)$. For general $g$, we briefly note that maximal subgroups in these classes can be treated similarly to tensor product groups (class $\mathcal{C}_7$) or constant groups in class $\mathcal{S}$ (class $\mathcal{C}_6$). In particular, going through the proofs of \Cref{prop_NoELargerThan1,prop: NoInducedTensors rewrite} one obtains a divisibility bound for those primes $\ell$ for which $G_\ell$ is contained in a subgroup of class $\mathcal{C}_7$.

\section{Example}\label{sect: Example}

Zywina \cite{ZywinaExample} has given an example of a three-dimensional Jacobian having maximal (adelic) Galois action, his approach consisting essentially in making effective a previous paper by Hall \cite{MR2820155} (see also the related work \cite{MR3501014}). Effective results based on Hall's techniques have also been obtained in \cite{MR3596606}, where a sufficient condition is given that implies that the equality $G_\ell=\operatorname{GSp}_{2g}(\mathbb{F}_\ell)$ holds for a given abelian variety and a fixed prime $\ell$. 
These results however assume that there is some place $v$ of $K$ where $A$ has potential toric rank $1$ at $v$ in order to force the existence of a transvection which eliminates the possibility of several classes of maximal subgroups. Our example is chosen precisely so as \textit{not} to satisfy this condition.
We focus on the genus $3$ hyperelliptic curve 
\begin{equation}\label{eq: example curve}
C: y^2+(x^4 + x^3 + x)y = -2x^6 - 4x^5 + 3x^4 + 6x^3 - 5x^2 - 4x + 3
\end{equation}
with Jacobian $J$. The curve $C$ appears in a database compiled by A.~Sutherland using the methods of \cite{MR3540958}. Using our methods, we shall show that the $\ell$-adic representation $\rho_{J,\ell^\infty}$ is surjective unless $\ell=2,7$. Using \Cref{prop: properties of example}(5) below and $q_v=2$, the bound of \Cref{thm_MainProof} gives surjectivity for all $\ell > 10^{1733}$ but we now refine this using the previous section. 

\begin{proposition}\label{prop: properties of example}
\begin{enumerate}
\item The torsion subgroup of $J(\mathbb{Q})$ is isomorphic to $\mathbb{Z}/14\mathbb{Z}$;
\item $J$ has good reduction away from $p_0=257$;
\item $J/\mathbb{Q}_{257}$ is semistable with toric rank $2$, conductor $257^2$ and Tamagawa number $1$;
\item $\op{End}_{\overline{\mathbb{Q}}} J = \mathbb{Z}$;
\item $h(J) \approx -2.2595232800...$
\end{enumerate}
\end{proposition}

\begin{proof}
\begin{enumerate}
\item This follows from the methods of \cite{MR4563690} with the corresponding code.
\item  The affine model has discriminant $257^2$ so $C$ has good reduction away from $257$.
\item  We change models to $y^2= x^8 + 2x^7 - 7x^6 - 14x^5 + 14x^4 + 24x^3 - 19x^2 - 16x + 12$. The statements then follow from \cite[Theorems 5.6, 12.1 and 10.1]{MR4453743} using the following cluster picture.

\begin{center}	
\clusterpicture
\Root(0.0,2)A(a1);     
\Root(0.4,2)A(a2);     
\Root(1.1,2)A(a3);      
\Root(1.5,2)A(a4);
\Root(2.2,2)A(a5);
\Root(2.6,2)A(a6);
\Root(3.0,2)A(a7);
\Root(3.4,2)A(a8);
\ClusterL(c1){(a1)(a2)}{$\tfrac{1}{2}$};
\ClusterL(c2){(a3)(a4)}{$\tfrac{1}{2}$};
\ClusterL(c3){(c1)(c2)(a5)(a6)(a7)(a8)}{$0$};  
\endclusterpicture			
\end{center}

\item The methods of \cite{MR3882288, MR3904148, MR4280568} show that $\operatorname{Jac}(C)$ is geometrically simple with trivial endomorphism ring. Note that since $\operatorname{Gal}(f(x)) \cong \operatorname{Sym}_3 \times C_4$ we cannot apply the criterion of \cite{MR1748293}.

\item This is a simple computation using \cite[Theorem 2.3]{MR3896852} and observing that the non-negative integer $e_{257}$ of \textit{loc.~cit.}~is necessarily zero, since (as part of the statement, using $v_{257}(\Delta_{\min}) \leq 2$) we have $2 \cdot 3 - (8 \cdot 3+4)e_{257} \geq 0$.
\end{enumerate}
\end{proof}

We now rule out the maximal subgroups. 

\begin{proposition}\label{prop: example is irreducible and primitive}
If $\ell\neq 2,7$, then $G_\ell$ is geometrically irreducible and primitive.
\end{proposition}

\begin{proof}
We first show irreducibility. Suppose not and let $W$ be a Jordan--H\"{o}lder factor of $J[\ell]$, which we may suppose to have dimension at most $3$. By \Cref{lemma: JHdet}(3) and \Cref{rmk: det}, $W$ has determinant $\chi_\ell^r$ for some $0 \leq r \leq \frac{1}{2}\dim W$. Taking $v \in \{2,3,7\}$, we use \Cref{lem: JHelim} to eliminate all possibilities except $\dim W=2$ and $\det W=\chi_\ell$.
To eliminate the remaining case, we use \Cref{lemma: eliminating two dimensional JH factors}. The only nontrivial divisor of $N_J$ is $p_0$, and taking $p=2$ and $p=3$ gives the desired conclusion.

Finally we show primitivity. Since $J$ is semistable, \Cref{lem: primEX}(3) reduces us to checking the primes $\ell=3,5$. In these cases, we apply \Cref{lem: primEX}(1) with $v=13, 151$.
\end{proof}

\begin{proposition}\label{prop: example no C3}
    If $\ell \neq 2,7$, then $G_\ell$ is not contained in a maximal subgroup of $\operatorname{GSp}_6(\mathbb{F}_\ell)$ of class $\mathcal{C}_3$.
\end{proposition}
\begin{proof}
Let $\ell \neq 2, 3, 7$ and suppose by contradiction that $G_\ell \subseteq M$, where $M$ acts $\mathbb{F}_{\ell^r}$-semilinearly on $J[\ell]$ for some prime $r \mid 2g=6$ (so $r\in\{2,3\}$).
    As $J$ is semistable, applying \Cref{lemma: ramification of C3 character}(2) we find that $A[\ell]$ is geometrically reducible, 
    contradicting \Cref{prop: example is irreducible and primitive}. 

    It remains only to consider $\ell=3$. If the character $\varphi$ defined in \Cref{lemma: ramification of C3 character} is trivial, $A[\ell]$ is geometrically reducible by \Cref{lemma: semilinear action}(2) and we conclude as above. Otherwise, $\varphi$ has order $2$ or $3$ and is ramified at most at $3$, so basic algebraic number theory shows that its kernel corresponds to $\mathbb{Q}(\sqrt{-3})$ or $\mathbb{Q}(\zeta_9)^+$. Any prime $p \equiv 5 \pmod{9}$ then satisfies $\SingleCharacter(\operatorname{Fr}_p) \neq 1$ for each of the relevant characters $\SingleCharacter$. Applying \Cref{lem: general phi}     
    with $v_\SingleCharacter=23$ for all $\SingleCharacter$ (the assumptions are satisfied by \Cref{rmk: conclusion in C3 case}), and then again with $v_\SingleCharacter=41$ for all $\SingleCharacter$, gives the conclusion.
\end{proof}

\begin{theorem}\label{thm: example}
Let $J$ be the Jacobian of the curve $C$ given in \eqref{eq: example curve}. Then $G_{\ell^\infty} \neq \op{GSp}_6 (\mathbb{Z}_\ell)$ if and only if $\ell \in \{2,7\}$.
\end{theorem}

\begin{proof}
By \Cref{cor_Conclusion} it suffices to show that $G_\ell \neq \op{GSp}_6 (\mathbb{F}_\ell)$ if and only if $\ell \in \{2,7\}$.
If $\ell=2, 7$, this follows from \Cref{prop: properties of example}(1). Let $\ell$ be any other prime and suppose that $G_\ell$ is contained in a maximal subgroup $M$ of $\operatorname{GSp}_6(\mathbb{F}_\ell)$. Recall that these are classified in \Cref{thm_Aschbacher} (note that $\det(G_\ell)=\mathbb{F}_\ell^\times$ since the $\bmod{\, \ell}$ cyclotomic character over $\mathbb{Q}$ is surjective for all primes $\ell$). By \Cref{prop: example is irreducible and primitive,prop: example no C3}, $M$ is not of class $\mathcal{C}_1$, $\mathcal{C}_2$ or $\mathcal{C}_3$. To rule out class $\mathcal{C}_4$ we apply \eqref{eq: divisibility relation for tensor product case} with $v=2, 3$. The sentence after \Cref{eq: divisibility for groups of Lie type} implies that $M$ cannot be the unique (up to conjugacy) maximal subgroup of class $\mathcal{S}$ of Lie type. Since $2\dim J$ is not a perfect power, the classes $\mathcal{C}_6$ and $\mathcal{C}_7$ cannot arise. It remains to consider the constant groups in class $\mathcal{S}$, which we handle as explained in \Cref{subsect: example constant groups}.
\end{proof}

\begin{remark}
    Combining \Cref{rmk: on ruling out the constant groups}(3) with \Cref{prop: properties of example}(3) rules out the case of constant groups of class $\mathcal{S}$ for all $\ell > 13$, $\ell \neq p_0$. However, the test described in \Cref{subsect: example constant groups} is fast enough that in our example we could test all primes $\ell$ up to $10^7$ with this method in a reasonable amount of time. For the vast majority of primes $\ell$, using $v=p=2$ was enough.
\end{remark}

\bibliography{Biblio}{}
\bibliographystyle{halpha-abbrv}

\end{document}